\documentclass{article}

\usepackage{xcolor}

\usepackage{amsfonts}
\usepackage{amsmath}
\usepackage{amssymb}
\usepackage{amsthm}
\usepackage{ifthen}

\newtheorem{theorem}{Theorem}[section]
\newtheorem*{MathieuConjecture}{Mathieu Conjecture}
\newtheorem*{VanishingConjecture}{Vanishing Conjecture}
\newtheorem*{GeneralizedVanishingConjecture}{Generalized Vanishing Conjecture}
\newtheorem{proposition}[theorem]{Proposition}
\newtheorem{corollary}[theorem]{Corollary}
\newtheorem{lemma}[theorem]{Lemma}
\theoremstyle{remark}
\newtheorem{remark}[theorem]{Remark}
\newtheorem{example}[theorem]{Example}
\theoremstyle{definition}
\newtheorem{definition}[theorem]{Definition}

\newcommand{\parder}[3][Default]{
	\frac{\partial \ifthenelse{\equal{#1}{Default}}{}{^{#1}}#2}{
              \partial #3 \ifthenelse{\equal{#1}{Default}}{}{^{#1}}}}

\newcommand{\lc}{\operatorname{lc}}

\newcommand{\imp}{{\mathversion{bold}$\Rightarrow$} }
\newcommand{\A}{{\mathcal A}}
\newcommand{\B}{{\mathcal B}}
\newcommand{\E}{{\mathcal E}}
\newcommand{\F}{{\mathcal F}}
\newcommand{\G}{{\mathcal G}}
\newcommand{\C}{{\mathbb C}}

\newcommand{\N}{{\mathbb N}}
\newcommand{\Z}{{\mathbb Z}}
\newcommand{\m}{{\mathfrak m}}
\newcommand{\p}{{\mathfrak p}}
\newcommand{\q}{{\mathfrak q}}
\newcommand{\rd}{{\mathfrak r}}

\makeatletter
\newcommand{\nolisttopbreak}{\vspace{\topsep}\nobreak\@afterheading}
\makeatother

\newenvironment{listproof}[1][\proofname]{\begin{proof}[#1]\mbox{}\nolisttopbreak}{\end{proof}}

\title{Some remarks on Mathieu subspaces over associative algebras}
\author{Michiel de Bondt}

\numberwithin{equation}{section}

\begin{document}

\maketitle

\begin{abstract}
\noindent
In \cite{msub}, the author takes a closer look at algebraic elements of radicals
of Mathieu subspaces (of associative algebras) over a field, and suggests to look at 
integral elements with rings other than fields. But it seems more useful to look at 
so-called co-integral elements. We generalize his theory about algebraic radicals 
over fields to co-integral radicals over commutative rings with unity. 

Furthermore, we show that over Artin rings, the concepts of integrality and co-integrality 
coincide.  In addition, we define so-called {\em uniform Mathieu subspaces},
inspired by the fact that Mathieu subspaces with co-integral radicals are always
of this type. Besides broadening the theory of \cite{msub} by means of the new concepts 
co-integrality and uniformity, we generalize many of the results of \cite{msub} in other 
ways as well. Furthermore, we obtain several new results.  

In the last section, we disprove a conjecture by the author of \cite{msub} (in a version of 
\cite{msub} prior to finding the counterexample), by showing that so-called strongly simple 
algebras do not need to be fields over theirselves.
\end{abstract}



\section{Introduction}

\cite{msub} is entitled `Mathieu subspaces over associative algebras', in which
the author W. Zhao introduces what the title expresses. The title of this paper
has the words `Some remarks on' in front, and can be seen as some remarks on 
the subject of \cite{msub}, as well as some remarks on \cite{msub} itself.
Let us first repeat the definition of Mathieu subspaces as formulated in \cite{msub}.

\begin{definition}[collecting {\cite[Def.\@ 1.1]{msub}} and {\cite[Def.\@ 1.2]{msub}}] 
\label{defmsub}
Let $M$ be an $R$-subspace ($R$-sub\-mod\-ule) of an associative $R$-algebra $\A$.
Then we call $M$ a {\em $\vartheta$-Mathieu subspace} of $\A$ if the following 
property holds for all $a, b, c \in \A$ such that $a^m \in M$ for all $m \ge 1$:
\begin{enumerate}
\item[(i)] $ba^m \in M$ when $m \gg 0$, if $\vartheta =$ ``{\it left}\/'';
\item[(ii)] $a^mc \in M$ when $m \gg 0$, if $\vartheta =$ ``{\it right}\/'';
\item[(iii)] $ba^m, a^mc \in M$ when $m \gg 0$, if $\vartheta =$ ``{\it pre-two-sided}\/'';
\item[(iv)] $ba^mc \in M$ when $m \gg 0$, if $\vartheta =$ ``{\it two-sided}\/''.
\end{enumerate}
\end{definition}

\noindent
If we replace every occurence of `$m \gg 0$' by `$m \ge 1$' in the above definition, 
then the cases ``{\it pre-two-sided}\/'' and ``{\it two-sided}\/'' coincide and we get 
the definition of a $\vartheta$ $R$-ideal of an associative $R$-algebra, 
or just the definition of a $\vartheta$ ideal of a ring since every ring is an 
associative $\Z$-algebra. 
Thus the concept of Mathieu subspaces is a generalization of that of ideals.

On the other hand, there is a single occurence of `$m \ge 1$' before the enumeration,
which can be replaced by `$m \gg 0$'.

\begin{proposition}[summarizing {\cite[Prop.\@ 2.1]{msub}}] \label{propmsub}
If we replace the single occurence of\/ {\upshape`$m \ge 1$'} in definition 
{\upshape\ref{defmsub}} by\/ {\upshape`$m \gg 0$'}, then we actually keep 
the same definition. 
\end{proposition}

\noindent
\iftrue
The definition of Mathieu subspace by Zhao was inspired by the following 
conjecture of O. Mathieu in \cite{Mathieu}. 

\begin{MathieuConjecture} Let $G$ be a compact Lie group with Haar
measure $\sigma$, and $f$ a complex-valued $G$-finite function on $G$ such that
$\int_G f^m {\mathrm d}\sigma=0$ for all $m \ge 1$. Then for each $G$-finite function
$g$ on $G$, we have that $\int_G gf^m{\mathrm d}\sigma=0$, for all $m \gg 0$.
\end{MathieuConjecture}

\noindent
This conjecture has been proved for the abelian case by
Duistermaat and van der Kallen in \cite{DvK}, and can be reformulated 
in terms of Mathieu subspaces, as follows.

\begin{MathieuConjecture} Let $G$ be a compact Lie group with Haar
measure $\sigma$, and $\A$ be the $\C$-algebra
of $G$-finite functions on $G$. Then the $\C$-subspace of $\A$ consisting of functions whose
integral over $G$ with respect to $\sigma$ is zero, is a Mathieu subspace of $\A$.
\end{MathieuConjecture}

\noindent
The Mathieu Conjecture resembles 
\cite[Conjecture 7.1]{vanish} by Zhao, which is given below, in both its structure and 
the fact that it implies the Jacobian conjecture.

\begin{VanishingConjecture} Let $P$ be a homogeneous
polynomial in $n$ variables over $\C$. If $\Delta^m(P^m)=0$ for all
positive $m$, then $\Delta^m(P^{m+1})=0$, for all $m$ large enough,
where $\Delta$ is the Laplace operator.
\end{VanishingConjecture}

\noindent
Inspired by the Mathieu Conjecture, the authors found the following even more resembling
equivalent formulation of the Vanishing Conjecture in \cite[Th.\@ 1.5]{twores}.

\begin{VanishingConjecture} Let $P$ be a homogeneous
polynomial in $n$ variables over $\C$. If $\Delta^m(P^m)=0$ for all
positive $m$, then $\Delta^m(Q P^m)=0$, for all $m$ large enough and all polynomials
$Q$ in $n$ variables over $\C$, where $\Delta$ is the Laplace operator.
\end{VanishingConjecture}

\noindent
Later on, Zhao dropped the homogeneity
condition on the polynomial $P$ in \cite{vanishdiff}, and replaced the Laplace operator
by any differential operator with constant coefficients, which resulted in:

\begin{GeneralizedVanishingConjecture} Let $\Lambda$ be any
differential operator with constant coefficients. If $P$ is a polynomial
over $\C$ such that $\Lambda^m(P^m)=0$ for all positive $m$, then for
any polynomial $Q$ over $\C$ we have that $\Lambda^m(QP^m)=0$, for all 
sufficiently large $m$.
\end{GeneralizedVanishingConjecture}

\noindent
Other related conjectures are the {\em Image Conjecture} and the {\em Dixmier
Conjecture}, the latter of which is actually equivalent to the Jacobian conjecture. 
See \cite{msub} and e.g.\@ \cite{images}, \cite{amazing}, and the references in all 
these papers.

\else
See \cite{msub} and the references therein for the motivation of the concept of 
Mathieu spaces.

\fi
Define the {\em radical} $\rd(M) = \sqrt{M}$ of $M$ as the set 
$\{a \in \A \mid a^m \in M$ for all $m \gg 0\}$, where $\A$ is the 
associative algebra at hand. 

If the $R$-subspace $M$ is an ideal of $\A$, then $\rd(M)$ is the 
usual radical of $M$, which is an ideal itself when $\A$ is commutative.
But $\rd(M)$ is not an ideal in general.
$\rd(M)$ does not even need to be a vector space over $R$ when $R$ is
a field. 

In proposition \ref{propmsub}, we can rewrite the condition that 
$a^m \in M$ for all $m \gg 0$ by $a \in \rd(M)$. This was actually done
in the original result \cite[Prop.\@ 2.1]{msub}. We will do the same 
in (ii) of \ref{propmsubC} below.

The radical of $M$ plays a crucial role in the theory of
Mathieu subspaces, and one can formulate the definition of Mathieu 
subspace entirely in terms of radicals, see \cite[Lm.\@ 2.3]{msub}.

In order to avoid distinguishing cases of $\vartheta$, we define a 
`constraint' $C_{\vartheta}(b,c)$ as follows.

\begin{definition} \label{defmsubC}
Let $\vartheta$ be any of the four types of Mathieu subspace and set
$$
C_{\vartheta}(b,c) := \left\{ \begin{array}{cl}
c = 1 & \mbox{if $\vartheta = \mbox{``{\it left}\/''}$} \\
b = 1 & \mbox{if $\vartheta = \mbox{``{\it right}\/''}$} \\
1 \in \{b,c\} & \mbox{if $\vartheta = \mbox{``{\it pre-two-sided}\/''}$} \\
1 + 1 = 2 & \mbox{if $\vartheta = \mbox{``{\it two-sided}\/''}$}
\end{array} \right.
$$
\end{definition}

\noindent
From definitions \ref{defmsub} and \ref{defmsubC} and proposition \ref{propmsub},
we obtain the following.

\begin{proposition} \label{propmsubC}
Let $M$ be an $R$-subspace ($R$-sub\-mod\-ule) of an associative $R$-algebra $\A$.
Then $M$ is a $\vartheta$-Mathieu subspace of $\A$, if and only if any of the following 
property holds:
\begin{enumerate}

\item[\upshape(i)] for all $a$ such that $a^m \in M$ for all $m \ge 1$ and all $b,c \in \A$ 
such that $C_{\vartheta}(b,c)$, we have $b a^m c \in M$
when $m \gg 0$;

\item[\upshape(ii)] for all $a \in \rd(M)$ and all $b,c \in \A$ 
such that $C_{\vartheta}(b,c)$, we have $b a^m c \in M$
when $m \gg 0$.

\end{enumerate}
\end{proposition}

\noindent
In the proof of \cite[Prop.\@ 2.5]{msub}, only the case $\vartheta = 
\mbox{``{\it left}\/''}$ is done, since the other cases are similar. This 
can be made precise using $C_{\vartheta}(b,c)$. Let us prove for example
the following reformulation of \cite[Prop.\@ 2.7]{msub} (without the condition 
that $V$ is an $R$-space).

\begin{proposition}[following {\cite[Prop.\@ 2.7]{msub}}] \label{surjhom}
Let $\A$ and $\B$ be associative $R$-algebras and $\phi: \A \rightarrow \B$
a {\em surjective} $R$-algebra homomorphism. Then $V \subseteq \B$ is a 
$\vartheta$-Mathieu subspace of $\B = \phi(\A)$, if and only if $\phi^{-1}(V)$ 
is a $\vartheta$-Mathieu subspace of $\A$.
\end{proposition}

\begin{proof}
The `only if'-part follows from \cite[Prop.\@ 2.5]{msub}, so assume that 
$\phi^{-1}(V)$ is a $\vartheta$-Mathieu subspace of $\A$. Since 
$\phi^{-1}(V)$ is an $R$-subspace of $\A$ and
$V = \phi(\phi^{-1}(V))$ by surjectivity of $\phi$, 
we see that $V$ is an $R$-subspace of $\B$. Take $a,b,c \in \B$ such
that $C_{\vartheta}(b,c)$ and $a^m \in V$ for all $m \ge 1$. 

Then there exist $a',b',c' \in \A$ such that $C_{\vartheta}(b',c')$, 
$\phi(a') = a$, $\phi(b') = b$ and $\phi(c') = c$.
Therefore $\phi((a')^m) = a^m \in V$ for all $m \ge 1$, i.e.\@
$(a')^m \in \phi^{-1}(V)$ for all $m \ge 1$. Since $\phi^{-1}(V)$ is a 
$\vartheta$-Mathieu subspace of $\A$, we have $b'(a')^m c' \in \phi^{-1}(V)$
when $m \gg 0$. Hence $b a^m c \in \phi(\phi^{-1}(V)) = V$ when $m \gg 0$.
\end{proof}

\section{Co-integral elements in the radicals of arbitrary subspaces}

We say that $a$ is co-integral over $R$ if $a^N R[a] = a^{N+1} R[a]$ for some $N \in \N$.
If $a$ is an invertible element of an associative algebra over $R$, the $a$ is co-integral 
over $R$, if and only if $a^{-1}$ is integral over $R$. This is because of the following,
which explains the choice of the the term co-integral, too.

The co-integrality condition $a^N R[a] = a^{N+1} R[a]$ is equivalent to that $a$ satisfies an 
algebraic relation over $R$ whose
trailing nonzero coefficient is equal to one, as opposed to the leading coefficient when $a$
is integral over $R$. So if $a$ is either integral or co-integral over $R$, 
say with corresponding algebraic relation $p \in R[t]$, then by decomposing $p = t^N q(t^{-1})$
for some $N \in \N$ and a $q \in R[t]$, we obtain the other for $a^{-1}$, namely with 
corresponding algebraic relation $q$.

If $a$ is co-integral over $R$ and $a^{-1} \notin \A$,
then one can show that $a$ is either a zero divisor in $R[a]$ or zero itself: 
take $V = R[a]$ in (ii) of proposition \ref{prop3.3} later in this section.

In this section, we extend results of the end of section 2 and of section 3
in \cite{msub}, mainly by generalizing from fields to commutative rings, 
replacing `integral' by `co-integral'. Additionally, we change the order and 
setup on some points. We may omit proofs if those of the original results are already
sufficient. 

\begin{definition} \label{coid}
Let $\A$ be an associative $R$-algebra.
Call $a$ {\em co-integral over $R$} if $a^N R[a] = a^{N+1} R[a]$ for some $N \in \N$.
Define $\rd'(V) := \{ a \in \rd(V) \mid \mbox{$a$ is co-integral over $R$}\}$
for $R$-subspaces $V$ of $\A$. Let $(a)_{\vartheta}$ be the $\vartheta$-ideal 
generated by $a$ when $\vartheta \ne$ ``{\it pre-two-sided}\/'', and 
$(a)_{\mbox{\scriptsize``{\it pre-two-sided}\/''}} = \A a + a \A$ be the sum of the
left and right ideals generated by $a$. Define $(S)_{\vartheta}$ in a similar manner
for subsets $S$ of $\A$. Then $(S)_{\vartheta} = \sum_{s \in S} (s)_{\vartheta}$.
\end{definition}

\noindent
In \cite{msub}, the definition of $\rd'(V)$ is different because `integral' is used 
instead of `co-integral'. This is because the author only uses $\rd'(V)$ when the base ring
is a field. In that case, the concept of co-integrality coincides with that of integrality,
so that the author can assume integrality and use co-integrality, which is the concept
that really matters. 

So we need to distinguish co-integrality and integrality for commutative base rings in general.
There are however rings other than fields for which both concepts coincide.

\begin{theorem} \label{ArtinK}
Assume $\A$ is an associative algebra over an Artin ring $R$. Take
$a \in \A$ arbitrary. Then $a$ is co-integral over $R$, if and only if $a$ is 
integral over $R$.
\end{theorem}

\begin{proof}
Let us first prove the `if'-part. For that purpose, suppose that
$a$ is integral over $R$. From \cite[Prop.\@ 5.1]{amd}, it follows that
$R[a]$ is a finitely generated $R$-module. From \cite[Prop.\@ 6.5]{amd},
it follows that $R[a]$ is an Artinian $R$-module. So we can apply the
descending chain condition on
$$
Ra + Ra^2 + Ra^3 + \cdots \supseteq Ra^2 + Ra^3 + \cdots \supseteq Ra^3 + \cdots \supseteq \cdots
$$
from which we deduce that $a$ is co-integral over $R$.

Hence the `only if'-part remains to be proved. So let $a$ be co-integral over $R$, 
i.e.\@ $a^NR[a] = a^{N+1}R[a]$ for some integer $N \ge 0$. We distinguish two cases.
\begin{itemize}
 
\item $N = 0$. \\
Then $1 = a^N \in a^{N+1}R[a] = aR[a]$, so $a$ is invertible. 
Consequently, $a^{-1}$ is integral over $R$ (see the second 
paragraph of this section). By the `if'-part, $a^{-1}$ is 
co-integral over $R$. So $a$ is integral over $R$ 
(see the second paragraph of this section again).

\item $N \ge 1$. \\
From lemma \ref{lm3.2} below, it follows that $a^N \in a^{2N}R[a]$, say that
$a^N = a^{2N} u(a)$, where $u \in K[t]$. Let $e = a^N u(a)$. Then 
$e a^N = a^N e = a^{2N} u(a) = a^N$, so $e a^m = a^m e = a^m$ for every $m \ge N$.
Consequently, $e^2 = e$. Now make the following definitions.
\begin{align*}
\tilde{R} &:= Re & \tilde{\A} &:= a^NR[a] & \tilde{a} &:= ea
\end{align*}
Since $e a^m = a^m e = a^m$ for every $m \ge N$, we see that 
$e \in \tilde{R}$ is unital in $\tilde{\A} \supseteq \tilde{R}$.
Hence $\tilde{\A}$ is an associative $\tilde{R}$-algebra by inclusion, 
with $e$ being the multiplicative unit of both $\tilde{R}$ and $\tilde{\A}$. 
Furthermore, $\tilde{R}$ is a homomorphic image of an Artin 
ring and hence Artinian itself. 

Since $R[a]$ is commutative, it follows
that $a^m = a^m e^m = \tilde{a}^m$ for every $m \ge N$, so
$$
\{\tilde{a},e\} \subseteq \tilde{A} = a^N R[a] = \tilde{a}^N \tilde{R}[\tilde{a}]
$$
It follows that $\tilde{A} = R[\tilde{a}]$ and that $e \in \tilde{a} \tilde{R}[\tilde{a}]$,
because $N \ge 1$. The inclusion $e \in \tilde{a} \tilde{R}[\tilde{a}]$ implies
that $\tilde{a}$ is co-integral over $\tilde{R}$, with the previous case 
$N = 0$ being in force. 

So $\tilde{a} \in \tilde{\A}$ is integral over 
$\tilde{R}$ on account of the previous case, i.e.\@ there is a polynomial $p$ over $\tilde{R}$
with leading coefficient $e$, such that $p(\tilde{a}) = 0$. Assume without loss of generality that
$p$ has no terms of degree less than $N$. Then $\tilde{a}^m = a^m e^m = a^m$ for all 
$m \ge N$ tells us that we can replace $e$ by $1$ and $\tilde{a}$ by $a$ without affecting 
$p(\tilde{a}) = 0$. So $a$ is integral over $R$. \qedhere

\end{itemize}
\end{proof}

\noindent
Notice that in the proof of the `only-if'-part of theorem \ref{ArtinK}, we only use the
validity of the `if'-part of theorem \ref{ArtinK} for homomorphic images of $R$. So any
ring $R$ with the property that the `if'-part of theorem \ref{ArtinK} is satisfied for 
homomorphic images of $R$, satisfies theorem \ref{ArtinK} as a consequence. 

\begin{lemma} \label{lm3.2}
Let $\A$ be an associative $R$-algebra and $a \in \A$.
If $a^N R[a] = a^{N+1} R[a]$, then $a^N \in a^m R[a]$ for every $m \ge 0$.
\end{lemma}

\begin{proof}
Take $m$ minimum, such that $a^N \notin a^{m+1} R[a]$. If $m \le N$, then
$$
a^N \in a^{N+1} R[a] \subseteq a^{N+1 - (N-m)} R[a] = a^{m+1} R[a]
$$
Consequently, $m > N$. Furthermore,
$$
a^N \in a^m R[a] = a^m R + a^{m+1} R[a]
$$
and
$$
a^m R \subseteq a^{N-m} a^N R[a] = a^{N-m} a^{N+1} R[a] = a^{m+1} R[a]
$$
This contradicts $a^N \notin a^{m+1} R[a]$, so $a^N \in a^m R[a]$ 
for every $m \ge 0$.
\end{proof}

\noindent
Lemma \ref{lm3.2} above can be seen as a replacement for \cite[Lm.\@ 3.3]{msub}.

Co-integrality is in some sense the opposite of integrality, just as Artinian is
in some sense the opposite of Noetherian. There is indeed a strong connection between
both pairs of opposite concepts.

\begin{proposition} \label{NoethAr}
Let $\A$ be an associative $R$-algebra $\A$ and $a \in \A$. Then we have the following.
\begin{enumerate}

\item[\upshape(i)] If $R$ is Noetherian, then $a$ is integral over $R$, if and only if $R[a]$ is
          Noetherian over $R$;

\item[\upshape(ii)] If $R$ is Artinian, then $a$ is co-integral over $R$, if and only if $R[a]$ is
           Artinian over $R$;

\end{enumerate}
\end{proposition}

\begin{proof}
We start with the `if'-parts of (i) and (ii). The `if'-part of (i) follows from the
fact that the integrality of $a$ over $R$ is just the ascending chain condition on
$$
Ra \subseteq Ra + Ra^2 \subseteq Ra + Ra^2 + Ra^3 \subseteq \cdots
$$
The `if'-part of (ii) follows from the fact that the co-integrality of $a$ over $R$ is 
just the descending chain condition on
$$
Ra + Ra^2 + Ra^3 + \cdots \supseteq Ra^2 + Ra^3 + \cdots \supseteq Ra^3 + \cdots \supseteq \cdots
$$
The `only if'-parts follow from \cite[Prop.\@ 5.1]{amd} and \cite[Prop.\@ 6.5]{amd}, 
except that we need that $a$ is integral over $R$ over $R$ instead of that $a$ is co-integral 
over $R$ in (ii). But that follows from theorem \ref{ArtinK}.
\end{proof}

\noindent
As opposed to the corresponding results in \cite{msub}, lemma \ref{lm3.1}
below also describes the situation where $a$ is not invertible. This allows
us to give a proof of theorem \ref{th3.89} which is more direct than the proof
of the corresponding results in \cite{msub}.

\begin{lemma}[following {\cite[Lm.\@ 3.1]{msub}} and 
{\cite[Lm.\@ 3.2]{msub}} more or less] \label{lm3.1}
Let $\A$ be an associative $R$-algebra and $V$ an $R$-subspace of $\A$.
Assume $a, b, c \in \A$ such that $b a^m c \in V$ when $m \gg 0$ and 
$a^N R[a] = a^{N+1} R[a]$ for some $N \in \N$. Then $b a^m c \in V$ for all $m \ge N$. 
If additionally $a$ is a unit in $\A$, then $b a^m c \in V$ for all $m \in \Z$. 
\end{lemma}

\begin{proof} 
Assume there is an $m \ge N$ ($m \in \Z$) such that $b a^m c \notin V$. 
Since $b a^m c \in V$ for all $m \gg 0$, there is a largest $m \ge N$ ($m \in \Z$) 
such that $b a^m c \notin V$. From $a^N R[a] = a^{N+1} R[a]$, we deduce that
$a^N c \in a^{N+1} R[a] c$, and multiplication with $b a^{m-N}$ (which requires that 
$a$ is a unit if $m < N$) gives $b a^m c \in b a^{m+1} R[a] c \subseteq V$. 
Contradiction, so $b a^m c \in V$ for all $m \ge N$ ($m \in \Z$).
\end{proof}

\begin{theorem}[combining {\cite[Th.\@ 3.9]{msub}} and {\cite[Th.\@ 3.10]{msub}} more or less]
\label{th3.89}
Let $\A$ be an associative $R$-algebra and $M$ be a $\vartheta$-Mathieu subspace of $\A$ over $R$. 
Suppose that $a \in \A$ and $N \ge 0$. Then for 
\begin{enumerate}

\item[\upshape(1)] $a \in \rd(M)$ and $a^N R[a] = a^{N+1} R[a]$;

\item[\upshape(2)] $(a^N)_{\vartheta} \subseteq M$;

\item[\upshape(3)] $a \in \rd(M)$;

\end{enumerate}
we have {\upshape(1)} $\Rightarrow$ {\upshape(2)} $\Rightarrow$ {\upshape(3)}.

In particular, $a \in \rd'(M)$ implies that $(a^N)_{\vartheta} \subseteq M$ for some $N$.
\end{theorem}

\begin{proof}
(2) $\Rightarrow$ (3) follows immediately from the definitions of
$\rd(M)$ and $(a^N)_{\vartheta}$. To prove (1) $\Rightarrow$ (2),
take $b,c \in \A$ such that $C_{\vartheta}(b,c)$, where $C_{\vartheta}(b,c)$
is as in definition \ref{defmsubC}. 

Since $a \in \rd(M)$ and $M$ is a $\vartheta$-Mathieu subspace of $\A$, we obtain by 
ii) of proposition \ref{propmsubC} that $b a^m c \in V$ for all $m \gg 0$.
As $a^N R[a] = a^{N+1} R[a]$, we deduce from lemma \ref{lm3.1} above that 
$b a^N c \in M$.

Since $b, c \in \A$ such that $C_{\vartheta}(b,c)$ were arbitrary, 
the desired result follows from definition \ref{coid}.
\end{proof}

\subsection{Idempotents of arbitrary subspaces}

Call $a$ a {\em semi-idempotent} if $a \in R a^2$, a {\em quasi-idempotent}
if $a \in R^{*} a^2$, and an {\em idempotent} if $a = a^2$. Here, $R^{*}$
denotes the set of units of $R$. The first of the three above definitions
does not appear in \cite{msub}. The other two are taken from \cite{msub}.

\begin{lemma}[following {\cite[Lm.\@ 2.9]{msub}} more or less] \label{lm2.9}
Let $a$ be a nonzero semi-idempotent of $\A$ and $V$ be an $R$-subspace of $A$. 
Then $a = r a^2$ for some $r \in R$. 
\begin{enumerate}

\item[\upshape(i)] 
If $a$ is nilpotent, then $a = 0$. \\
If $a \ne 0$ is not a zero divisor in $R[a]$, then $a = (r \cdot 1)^{-1}$, which is a unit. \\
If $a \ne 0$ is not a zero divisor in $R[a]$ and $a$ is an idempotent, then $a = 1$.

\item[\upshape(ii)] $a \in \rd(V)$ implies $a^m \in V$ for all $m \ge 1$.

\item[\upshape(iii)] If $a \in V$ is a (quasi-)idempotent, then $a \in \rd(V)$. \\
If $a \in \rd(V)$ and $V$ is a $\vartheta$-Mathieu subspace 
of $\A$, then $(a)_{\vartheta} \subseteq V$.

\end{enumerate}
\end{lemma}

\begin{proof} Since $a \in R a^2$, we can write $a = r a^2$ with $r \in R$. 
\begin{enumerate}

\item[(i)] If $a$ is nilpotent, say that $a^m = 0$, then $a = r a^2 = r^2 a^3 = \cdots = 
r^{m-1} a^m = 0$ indeed. Furthermore, we obtain from $a = r a^2$ that $a(1-ra) = 0$. Thus
if $a \ne 0$ is not a zero divisor in $R[a]$, then $ra = 1$, whence $a = (r \cdot 1)^{-1}$. 
If additionally $a$ is an idempotent, then we can take $r = 1$, so that $a = 1^{-1} = 1$.

\item[(ii)] Assume $a \in \rd(V)$. Since $a \in R a^2$, we have $a R[a] = a^2 R[a]$.
Hence by lemma \ref{lm3.1} with $b = c = 1 = N$, $a^m \in V$ for all $m \ge 1$, as desired.

\item[(iii)] If $a \in V$ is a quasi-idempotent, then we can take $r \in R^{*}$, so that
$a^m = r^{1-m}a \in V$ for all $m \ge 1$. In particular, $a \in V$ implies $a \in \rd(V)$. 
Next, assume that $V$ is a Mathieu-subspace of $\A$ and $a \in \rd(V)$. 
Since $a \in a^2 R[a]$, we have $a R[a] = a^2 R[a]$. Hence $(a)_{\vartheta} \subseteq V$
on account of (1) $\Rightarrow$ (2) or theorem \ref{th3.89} with $N = 1$. \qedhere

\end{enumerate}
\end{proof}

\begin{proposition}[following {\cite[Prop.\@ 3.4]{msub}} more or less] \label{prop3.3}
Let $V$ be an $R$-subspace of $\A$ and $a \in \rd'(V)$. Take $N$ such that 
$a^N R[a] = a^{N+1} R[a]$. Then there exists an idempotent $e \in a^N R[a] \subseteq V$ such 
that $a^N = a^Ne = ea^N$.

Additionally, we have
\begin{enumerate}
\item [\upshape(i)] $a$ is nilpotent, if and only if $a^N = 0$, if and only if $e = 0$,
\item [\upshape(ii)] $a$ is a unit in $\A$, if and only if $a \ne 0$ is not a zero divisor in $R[a]$,
if and only if $e = 1$.
\end{enumerate}
\end{proposition}

\begin{proof}
From lemma \ref{lm3.1} with $b = c = 1$, it follows that $a^m \in V$ for all 
$m \ge N$. Hence $a^N R[a] \subseteq V$. If $N = 0$, then we take $e = 1$.
If $N \ge 1$, then we take $e$ as in the case $N \ge 1$ in the proof of theorem 
\ref{ArtinK}. In both cases, $e \in a^N R[a]$ and $a^N = a^Ne = ea^N$.
\begin{enumerate}

\item[(i)] 
Since $a^N = a^Ne = ea^N$, we see that $e = 0$ implies $a^N = 0$. Conversely
$e \in a^N R[a]$ tells us that $a^N = 0$ implies $e = 0$. Thus it remains
to show that $a^N = 0$ in case $a$ is nilpotent, i.e.\@ $a \in \rd'(0)$. 
This follows from $a^N R[a] \subseteq V$, because we can take $V = 0$ when $a \in \rd'(0)$.

\item[(ii)]
If $a \ne 0$ is not a zero divisor in $R[a]$, then we can cancel $a^N$ everywhere
in $a^N = a^Ne = ea^N$, which
gives $e = 1$. Since units are not zero divisors, it remains to show that
$a$ is a unit in case $e = 1$. Hence assume that $e = 1$. If $N \ge 1$, then 
$e \in t^N R[t]$
tells us that $a \mid a^N \mid e = 1$. If $N = 0$, then $1 = a^N \in a^N R[a] = 
a^{N+1} R[a] = a R[a]$, which leads to $a \mid 1$ as well, as desired. \qedhere

\end{enumerate}
\end{proof}

\begin{theorem}[following {\cite[Th.\@ 3.5]{msub}} more or less] \label{th3.4}
Let $\A$ be an associative $R$-algebra and $V$ an $R$-subspace of $\A$. Then
the following statements are equivalent.
\begin{enumerate}
\item[\upshape(1)] Every non-unit of $\rd'(V)$ is nilpotent.
\item[\upshape(2)] Every zero divisor of $\rd'(V)$ is nilpotent. 
\item[\upshape(3)] $V$ contains no idempotents other than $0$ and $1$.
\end{enumerate}
\end{theorem}

\begin{proof} Since (1) $\Rightarrow$ (2) follows from the fact that zero divisors are 
non-units, the following remains to be proved.
\begin{description}

\item[(2) \imp (3)] Assume $V$ contains an idempotent $e \notin\{0,1\}$.
Then by (i) of lemma \ref{lm2.9}, $e$ is a zero divisor because $e \ne 1$, but 
additionally $e$ is not nilpotent because $e \ne 0$.
This gives the desired result.

\item[(3) \imp (1)] Assume $a \in \rd'(V)$ is a non-unit, but not nilpotent.
By proposition \ref{prop3.3}, $R[a]$ contains an idempotent $e$, for which 
$e \ne 0$ and $e \ne 1$ on account of (i) and (ii) of proposition \ref{prop3.3}
respectively. This gives the desired result. \qedhere

\end{description}
\end{proof}

\noindent
If we take $V = \A = \rd(\A)$ in the above theorem (just as in \cite{msub}), 
we obtain the following.

\begin{corollary}[following {\cite[Cor.\@ 3.6]{msub}} more or less] \label{cor3.5}
For every associative $R$-algebra $\A$, the following statements are equivalent.
\begin{enumerate}
\item[\upshape(1)] Every non-unit of $\rd'(\A)$ is nilpotent.
\item[\upshape(2)] Every zero divisor of $\rd'(\A)$ is nilpotent. 
\item[\upshape(3)] $\A$ contains no idempotents other than $0$ and $1$.
\end{enumerate}
\end{corollary}

\begin{lemma}[same as {\cite[Lm.\@ 3.7]{msub}}] \label{lm3.6}
Let $\A$ be an associative $R$-algebra. Then for the following three statements:
\begin{enumerate}
\item[\upshape(1)] every non-unit of $\A$ is nilpotent;
\item[\upshape(2)] $\A$ is a local $R$-algebra;
\item[\upshape(3)] $\A$ contains no idempotents other than $0$ and $1$;
\end{enumerate}
we have {\upshape(1)} $\Rightarrow$ {\upshape(2)} $\Rightarrow$ {\upshape(3)}.
\end{lemma}
\iftrue

\begin{listproof}[Proof (somewhat more direct than the original proof)]
\begin{description}
\item[(1) \imp (2)] It is a nice exercise for the reader to show that the 
nilpotent elements of $\A$ form an ideal if (1) holds. Hence (1) implies (2).

\item[(2) \imp (3)] 
Assume $\A$ is local and $\A$ has an idempotent $e$. 
Then $e (1 - e) = 0 = (1 - e) e$. 
Since $e$ and $1 - e$ cannot be contained in the same proper ideal
of $\A$ and $\A$ is local, one of $e$ and $1 - e$ must be a unit.
The other must be zero, because units are not zero divisors and 
$e (1 - e) = 0 = (1 - e) e$. So $e \in \{0,1\}$. \qedhere

\end{description}
\end{listproof}
\fi

\subsection{Quasi-stable algebras}

The following definition appears at the beginning of section 7 in \cite{msub},
which has the same title as this subsection. 

\begin{definition}
Let $\A$ be an associative $R$-algebra. We say that $\A$ is $\vartheta$-quasi-stable
(or $\vartheta$-stable), if every $R$-subspace $V$ of $\A$ with $1 \notin V$ is
a $\vartheta$-Mathieu subspace of $\A$ (or a $\vartheta$-ideal of $\A$ respectively).
\end{definition}

\noindent
If we combine lemma \ref{lm3.6} with (3) $\Rightarrow$ (1) of corollary 
\ref{cor3.5} (just as in \cite{msub}), then we get the first assertion
in the following. 

\begin{corollary}[following {\cite[Cor.\@ 3.8]{msub}} and {\cite[Prop.\@ 7.4]{msub}}] 
\label{cor3.7}
For every associative $R$-algebra $\A$ with $\rd'(\A) = \A$, the three statements in 
lemma \ref{lm3.6} are equivalent. Furthermore, $\A$ is (two-sided) quasi-stable 
over $R$ in case any of these three statements is fulfilled.
\end{corollary}

\begin{proof}
Assume that the equivalent statements of lemma \ref{lm3.6} are fulfilled.
Let $V$ be an $R$-subspace of $\A$ such that $1 \not\in V$. 
It suffices to show that $\rd(V) \subseteq \rd((0))$.
So assume $a \in \rd(V) = \rd'(V)$ such that $a \not\in \rd((0))$.
On account of the first statement of lemma \ref{lm3.6}, $a$ is invertible over $R$. 
Hence $a^m \in V$ for all $m \in \Z$ by lemma \ref{lm3.1}. 
This contradicts $1 \not\in V$, thus $a \in \rd((0))$. 
\end{proof}

\noindent
A generalization which applies to both integrality and co-integrality 
of $a$, is that some coefficient of the polynomial which has $a$ as a root
must be equal to $1$, but not necessarily the leading or trailing nonzero 
coefficient. 

If $a$ is invertible, then we can shift this polynomial
to obtain a Laurant polynomial which has $1$ as constant term.
In other words, $1 \in aR[a] + a^{-1}R[a^{-1}]$.

Similarly, an invertible element $a \in A$ is co-integral or integral
over $R$, if and only if $1 \in aR[a]$ or $1 \in a^{-1}R[a^{-1}]$
respectively.

\begin{proposition}[generalizing {\cite[Prop.\@ 7.4]{msub}}] 
Let $\A$ be an associative $R$-algebra, such that every element of $\A$ is
either invertible or nilpotent, and every invertible element $a \in \A$
satisfies $1 \in aR[a] + a^{-1}R[a^{-1}]$. Then $\rd'(\A) = \A$ and $\A$
is integral and (two-sided) quasi-stable over $R$.
\end{proposition}
\iftrue

\begin{proof}
We first show that each $a \in \A$ is both integral and co-integral over $R$. So take
$a \in \A$ arbitrary. If $a$ is nilpotent, say that $a^m = 0$, then clearly
$a$ is integral over $R$ and $a^m R[a] = a^{m+1}R[a]$. Hence suppose that $a$ is not nilpotent.
Then $a$ is invertible and $1 \in aR[a] + a^{-1}R[a^{-1}]$ by assumption.
Say that $1 = f(a)$ where $f(z)$ is a Laurant polynomial without
constant term. Let $\tilde{f}(z)$ be the Laurant polynomial consisting
of the terms $r_i z^i$ of $f$ such that $r_i \cdot 1$ is not nilpotent.
Then $\tilde{f}$ has no constant term either. 

Since the nilpotent elements of the commutative algebra $R[a,a^{-1}]$ form an ideal 
of $R[a,a^{-1}]$, we have that $1 - \tilde{f}(a) = 
f(a) - \tilde{f}(a)$ is nilpotent. Hence $\tilde{f}(a) = 1 - (1 - \tilde{f}(a))$
is invertible in $R[a,a^{-1}]$ and
$$
1 = \frac{\tilde{f}(a)}{1 - (1 - \tilde{f}(a))} = 
\tilde{f}(a)\sum_{i=0}^{k} (1 - \tilde{f}(a))^i
$$
for some $k \in \N$.

Notice that the leading term of $\tilde{f}(z)\sum_{i=0}^{k} 
(1 - \tilde{f}(z))^i - 1$ is $-1$ in case $\tilde{f}(z) \in R[z^{-1}]$.
Since every nonzero coefficient of $\tilde{f}$ is invertible, the leading
nonzero coefficient of $\tilde{f}(z)\sum_{i=0}^{k} 
(1 - \tilde{f}(z))^i - 1$ is invertible, regardles of whether
$\tilde{f}(z) \in R[z^{-1}]$ or not. Hence $a$ is integral over $R$. 

Similarly, the trailing nonzero of $\tilde{f}(z)\sum_{i=0}^{k} 
(1 - \tilde{f}(z))^i - 1$ is invertible, and $a$ is co-integral over $R$. 
Thus $\A$ is integral over $R$, $\rd'(\A) = \A$, and 
by corollary \ref{cor3.7}, $\A$ is two-sided quasi-stable over $R$.
\end{proof}
\noindent
\else

\noindent
It is a nice puzzle for the reader to show that every (invertible) element
of $\A$ is actually both integral and co-integral in the above theorem, 
after which the proof follows directly from corollary \ref{cor3.7}.

\fi
For more results about quasi-stable algebras, see section 7 of \cite{msub}.

\subsection{Localization of the base ring}

We end this section with some results about integrality, co-integrality and localization 
of the base ring. First, we formulate results about co-integrality and localization 
of the base ring.

\begin{proposition}
Assume $\A$ is an associative $R$-algebra, and $S \ni 1$ is a multiplicatively closed
subset of $R$.
\begin{enumerate}

\item[\upshape(i)]
If $s a^N \in a^{N+1} R[a]$ for some $s \in S$, then $s^{-1}a$ is co-integral over $R$.

\item[\upshape(ii)]
If $a \in \A$ is co-integral over $R$, then for all $s, s' \in S$,
$s^{-1} a$ is co-integral over $R$ and $s^{-1} s' a$ is co-integral over $S^{-1} R$.

\item[\upshape(iii)]
If $b \in S^{-1} A$ is co-integral over $R$, then for all $s, s' \in S$,
$s^{-1} b$ is co-integral over $R$ and $s^{-1} s' b$ is co-integral over $S^{-1} R$.

\item[\upshape(iv)]
If $b \in S^{-1} A$ is co-integral over $S^{-1} R$, then there exists
an $s \in S$ such that $s^{-1} b$ is co-integral over $R$. Furthermore,
$s^{-1} s' b$ is co-integral over $S^{-1} R$ for all $s, s' \in S$.

\end{enumerate}
\end{proposition}
\iftrue

\begin{listproof}
\begin{enumerate}

\item[(i)] Multiplication of $s a^N \in a^{N+1} R[a]$ by $s^{-N-1}$ gives
$(s^{-1}a)^N = s^{-N} a^N \in s^{-N-1} a^{N+1} R[a] \subseteq (s^{-1}a)^{N+1} R[s^{-1}a]$.

\item[(ii)] Multiplication of $a^N \in a^{N+1} R[a]$ by $s^{-N}$ gives
$(s^{-1}a)^N = s^{-N} a^N \in s^{-N-1} a^{N+1} s R[a] \subseteq (s^{-1}a)^{N+1} R[s^{-1}a]$.
$(s^{-1}a)^N \in (s^{-1}a)^{N+1} R[s^{-1}a]$ in turn can be multiplied by $(s')^N$, to obtain
the second claim.
 
\item[(iii)] Replace $a$ by $b$ in the proof of (ii).

\item[(iv)] Say that $b^N = b^{N+1} p(b)$ for some univariate polynomial $p$ over $S^{-1}R$.
Let $s$ be the product of the denominators of the coefficients of $p$. Then
$(s^{-1}b)^N = s^{-N} b^N = s^{-N-1} b^{N+1} s\, p(b) \in (s^{-1}b)^{N+1} R[s^{-1}b]$. The second claim
follows is a similar manner as the second claim in (iii). \qedhere

\end{enumerate}
\end{listproof} 
\noindent
\else

\noindent
The proof is straightforward and left as an exercise to the reader.

\fi
Although co-integrality seems a more useful concept than integrality in this context,
(v) of the next theorem is about integrality and localization of the base ring.
 
\begin{theorem} \label{locint}
Assume $\A$ is an associative $R$-algebra, and $S \ni 1$ is a multiplicatively closed
subset of $R$. Write $\phi: \A \rightarrow S^{-1}\A$ for the localization map.
\begin{enumerate}

\item[\upshape(i)] If $M$ is a $\vartheta$-Mathieu subspace over $S^{-1} R$ of $S^{-1}\A$, 
then $M$ is also a $\vartheta$-Mathieu subspace over $R$ over $S^{-1}\A$.

\item[\upshape(ii)] If $M$ is a $\vartheta$-Mathieu subspace over $R$ of $S^{-1}\A$, then 
$\phi^{-1}(M)$ is a $\vartheta$-Mathieu subspace over $R$ of $\A$.

\item[\upshape(iii)] If $V \subseteq S^{-1}\A$, then $M := \phi^{-1}(V)$ is a 
$\vartheta$-Mathieu subspace over $R$ of $\A$, if and only if $\phi(M)$ 
is a $\vartheta$-Mathieu subspace over $R$ of $\phi(\A)$.

\item[\upshape(iv)] Assume that $V \subseteq S^{-1}\A$ and 
$M := \phi^{-1}(V)$ is a $\vartheta$-Mathieu subspace over $R$ of $\A$,
such that for each $a \in S^{-1}\A$ such that $a^m \in S^{-1}M$ for all $m \ge 1$, 
there exists an $s \in S$ such that $(sa)^m \in \phi(M)$ for all $m \ge 1$.
Then $S^{-1}M$ is a $\vartheta$-Mathieu subspace over $S^{-1}R$ of $S^{-1} \A$. 

\item[\upshape(v)] For a specific $a$ as in {\upshape(iv)}, i.e.\@ 
$a^m \in S^{-1}M$ for all $m \ge 1$, an $s$ as in {\upshape(iv)} exists 
in case $s'a$ is integral over $R$ (not necessary co-integral) for some 
$s' \in S$. 

\end{enumerate}
\end{theorem}
\iftrue

\begin{listproof}
\begin{enumerate}

\item[(i)] This follows from the trivial fact that an $S^{-1} R$-subspace is 
also an $R$-subspace.

\item[(ii)] This follows from \cite[Prop.\@ 2.5]{msub}.

\item[(iii)] Since $\phi^{-1}(\phi(\phi^{-1}(V))) = \phi^{-1}(V)$,
we have $\phi^{-1}(\phi(M)) = M$. Hence taking $V = \phi(M)$ in proposition \ref{surjhom}
gives the desired result. 

\item[(iv)]
Take any element $a \in S^{-1}\A$ such that $a^m \in S^{-1}M$ for all $m \ge 1$.
By assumption, there exists an $s \in S$ such that $(sa)^m \in \phi(M)$ 
for all $m \ge 1$. By (iii), we see that $\phi(M)$ is a $\vartheta$-Mathieu 
subspace over $R$ of $\phi(\A)$. Thus for all $b',c' \in \phi(\A)$ such that 
$C_{\vartheta}(b',c')$, we have $b' (sa)^m c' \in \phi(M)$ for all $m \gg 0$, 
where $C_{\vartheta}(b',c')$ is as in definition \ref{defmsubC}. 

Consequently, for all $b',c' \in \phi(\A)$ such that $C_{\vartheta}(b',c')$, we have
$b' a^m c' \in S^{-1}M$ for all $m \gg 0$. For all $b,c \in S^{-1}M$, there exists an 
$s' \in S$ such that $s' b, s' c \in \phi(\A)$. Using this fact as far as $b \ne 1 \ne c$,
we deduce that for all $b,c \in S^{-1}M$ such that $C_{\vartheta}(b,c)$, we have 
$b a^m c \in S^{-1}M$ for all $m \gg 0$ as well. Hence it follows from (i) of proposition 
\ref{propmsubC} that $S^{-1}M$ is a $\vartheta$-Mathieu subspace over $S^{-1}R$ of 
$S^{-1} \A$. 

\item[(v)]
Since $a^m \in S^{-1}M$, there exist $s_m \in S$ such that $s_m a^m \in \phi(M)$
for all $m \ge 1$. Consequently, $(s_1 s_2\cdots s_d s' a)^m \in \phi(M)$ 
for all $d \in \N$, all $m$ with $1 \le m \le d$ and all $s' \in S$.

By assumption, there exists an $s' \in S$, and a monic $f \in R[t]$, say of degree
$d$, such that $f(s'a) = 0$. Therefore $(s'a)^m \in 
R \cdot (s'a) + R \cdot (s'a)^2 + \cdots + R \cdot (s'a)^d$ follows inductively for all $m > d$.
By multiplication by $(s_1 s_2 \ldots s_d)^m$ on both sides, we see that
$(s_1 s_2\cdots s_d s'a)^m \in \phi(M)$ for all $m > d$ as well.
Thus $(s_1 s_2\cdots s_d s'a)^m \in \phi(M)$ for all $m \ge 1$,
i.e.\@ $s = s_1 s_2\cdots s_d s'$ suffices.
\qedhere

\end{enumerate}
\end{listproof}
\else

\noindent
The proof is left as an exercise to the reader.
\fi

\section{Uniform Mathieu subspaces}

In this section, we generalize results of section 4 of \cite{msub}, which
is entitled `Mathieu subspaces with algebraic radicals'. Hence you might
expect a section about Mathieu subspaces with co-integral radicals, but it
appears that such Mathieu subspaces are so-called uniform Mathieu subspaces,
see theorem \ref{strongE} below. 

\begin{definition}
\label{defsmsub}
Let $M$ be an $R$-subspace ($R$-sub\-mod\-ule) of an associative $R$-algebra $\A$.
Then we call $M$ a {\em uniform $\vartheta$-Mathieu subspace} of $\A$ if for all 
$a \in \A$ such that $a^m \in M$ for all $m \ge 1$, there exists an $N \in \N$ such that 
$(a^N)_{\vartheta} \subseteq M$.
\end{definition}

\begin{proposition} \label{smsubmsub}
If $M$ is a uniform $\vartheta$-Mathieu subspace of an associative $R$-algebra $\A$,
then $M$ is also a $\vartheta$-Mathieu subspace of $\A$.
\end{proposition}

\begin{proof}
Assume $M$ is a uniform $\vartheta$-Mathieu subspace of an associative $R$-algebra $\A$
and $a^m \in M$ for all $m \ge 1$. Then there exists an $N \in \N$ such that 
$(a^N)_{\vartheta} \subseteq M$, and we have the following when $m \ge N$:
\begin{enumerate}
\item[\upshape(i)] $ba^m \in M$ for all $b \in \A$, if $\vartheta =$ ``{\it left}\/'';
\item[\upshape(ii)] $a^mc \in M$ for all $c \in \A$, if $\vartheta =$ ``{\it right}\/'';
\item[\upshape(iii)] $ba^m, a^mc \in M$ for all $b,c \in \A$, if $\vartheta =$ ``{\it pre-two-sided}\/'';
\item[\upshape(iv)] $ba^mc \in M$ for all $b,c \in \A$, if $\vartheta =$ ``{\it two-sided}\/''.
\end{enumerate}
Hence $M$ is a $\vartheta$-Mathieu subspace of $\A$ on account of definition
\ref{defmsub}.
\end{proof}

\noindent
Notice that the difference between Mathieu subspace and uniform Mathieu subspaces is that for
uniform Mathieu subspaces, the number $N$ is the above proposition does not depend on the elements 
$b$ and/or $c$ of $\A$, as opposed to regular Mathieu subspaces.

The following propositions are analogs for uniform Mathieu subspaces of \cite[Prop.\@ 2.1]{msub}
and proposition \ref{propmsubC} respectively.

\begin{proposition}[following {\cite[Prop.\@ 2.1]{msub}} more or less] 
\label{propsmsub}
Let $M$ be an $R$-subspace ($R$-sub\-mod\-ule) of an associative $R$-algebra $\A$.
Then $M$ is a uniform $\vartheta$-Mathieu subspace of $\A$, if and only if 
for all $a \in \rd(M)$, $(a^N)_{\vartheta} \subseteq M$ for some $N \in \N$.
\end{proposition}

\begin{proof}
In order to prove the `only-if'-part, assume that $M$ is a uniform 
$\vartheta$-Mathieu subspace of $\A$ and $a \in \rd(M)$. Then there exists
a $k \in \N$ such that $(a^k)^m \in M$ for all $m \ge 1$. Hence
$(a^{kN})_{\vartheta} = ((a^k)^N)_{\vartheta} \subseteq M$ for some $N \in \N$. 
This gives the `only-if'-part.

Since $a^m \in M$ for all $m \ge 1$ implies $a \in \rd(M)$, the `if'-part
follows as well.
\end{proof}

\begin{proposition} \label{propsmsubC}
Let $M$ be an $R$-subspace ($R$-sub\-mod\-ule) of an associative $R$-algebra $\A$.
Then $M$ is a uniform $\vartheta$-Mathieu subspace of $\A$, if and only if any of 
the following properties holds, where $C_{\vartheta}(b,c)$
is as in definition \ref{defmsubC}:
\begin{enumerate}

\item[\upshape(i)] for all $a$ such that $a^m \in M$ for all $m \ge 1$, we have the following
when $m \gg 0$: $b a^m c \in M$ for all $b,c \in \A$ such that $C_{\vartheta}(b,c)$;

\item[\upshape(ii)] for all $a \in \rd(M)$, we have the following
when $m \gg 0$: $b a^m c \in M$ for all $b,c \in \A$ such that $C_{\vartheta}(b,c)$.

\end{enumerate}
\end{proposition}

\begin{proof}
Comparing the proof of proposition \ref{smsubmsub} with definition \ref{defmsub},
we see that an alternative definition of uniform Mathieu subspace can be obtained
by interchanging the quantification with $m$ with that of $b$ and/or $c$ in the 
definition of Mathieu subspace as given in definition \ref{defmsub}.
Since this proposition and a possible proof differs accordingly from 
proposition \ref{propmsubC} and its proof respectively,
the desired result follows.
\end{proof}

\begin{remark} \label{smrem}
The proof of proposition \ref{propsmsubC} tells us how the proofs of several 
results about Mathieu subspaces can be turned into similar proofs for uniform
Mathieu subspaces. Results with an analog for uniform Mathieu subspaces that can 
be proved in this manner are \cite[Prop.\@ 2.5--Lm.\@ 2.8]{msub}, proposition
\ref{surjhom} and theorem \ref{locint}.
\end{remark}

\mathversion{bold}
\subsection{Definitions of $\G_{\vartheta}(\A)$ and $\E_{\vartheta}(\A)$}
\mathversion{normal}

In \cite{msub}, $\G(\A)$ is defined as the set of all $K$-subspaces
$V$ of $\A$ such that $\rd'(V) = \rd(V)$, where $K = R$ is a field. 
Before we give another definition of $\G(\A)$, we formulate a proposition.
Recall that for subsets $S$ of $\A$, $(S)_{\vartheta}$ is the $\vartheta$-ideal 
generated by $S$ when $\vartheta \ne$ ``{\it pre-two-sided}\/'', and 
$(S)_{\mbox{\scriptsize``{\it pre-two-sided}\/''}} = \A S + S \A$ is the sum of the
left and right ideals generated by $S$. 

\begin{proposition} \label{rdce}
Let $\A$ be an associative $R$-algebra and $V$ an $R$-subspace of $\A$. Then
$$
\rd'(V) \subseteq \rd\big((e \in V \mid e^2 = e)_{\vartheta}\big)
$$
\end{proposition}

\begin{proof}
Take $a \in \rd'(V)$. From proposition \ref{prop3.3}, it follows that there exist
an $N \in \N$ and an idempotent $e \in V$ such that 
$a^N = a^N e = e a^N \in (e)_{\vartheta}$. 
Hence $a \in \rd\big((e \in V \mid e^2 = e)_{\vartheta}\big)$.
\end{proof}

\noindent
By proposition \ref{rdce}, condition \eqref{Gdefeq} in the definition below is weaker than
the condition $\rd'(V) = \rd(V)$ in \cite{msub}.

\begin{definition}
Let $\A$ be an associative $R$-algebra and $\vartheta \ne$ {\it{``pre-two-sided''}}. 
Then we define $\G_{\vartheta}(\A)$ as the set of all $R$-subspaces $V$ of $\A$, such that
\begin{equation}
\rd(V) \subseteq \rd\big((e \in V \mid e^2 = e)_{\vartheta}\big) \label{Gdefeq}
\end{equation}
Since the pre-two-sided case is a combination of both one-sided cases, we simply define
$$
\G_{\mbox{\scriptsize``{\it pre-two-sided}\/''}}(\A) :=
\G_{\mbox{\scriptsize``{\it left}\/''}}(\A) \cap
\G_{\mbox{\scriptsize``{\it right}\/''}}(\A)
$$
Let $\E_{\vartheta}(\A)$ be the subset of $\vartheta$-Mathieu subspaces of $\G_{\vartheta}(\A)$.
\end{definition}

\noindent
We will show in corollary \ref{cor4.13}, which follows later, that in the commutative case, 
$\E_{\vartheta}(\A)$ is just the set of all $R$-subspaces $V$ of $\A$ for which we have equality 
in \eqref{Gdefeq}.

The following theorem gives another definition of $\G_{\vartheta}(\A)$
for the commutative case, namely $$\rd(V) \subseteq \rd\big((\rd'(V))_{\vartheta}\big)$$
instead of \eqref{Gdefeq}, because $\rd\big((e \in V \mid e^2 = e)_{\vartheta}\big)$ is an ideal 
when $\A$ is commutative.

\begin{theorem}
Let $\A$ be an associative $R$-algebra and $V$ an $R$-subspace of $\A$. Suppose that
$(e \in V \mid e^2 = e)_{\vartheta}$ is a $\vartheta$-Mathieu subspace 
(which is obviously the case when $\vartheta \neq$ {\it{``pre-two-sided''}}). 
Suppose additionally that either $\rd\big((e \in V \mid e^2 = e)_{\vartheta}\big)$ or 
$\rd'(V)$ is a $\vartheta$-ideal of $\A$. Then
$$
\rd\big((e \in V \mid e^2 = e)_{\vartheta}\big) = \rd\big((\rd'(V))_{\vartheta}\big)
$$
\end{theorem}

\begin{proof}
Let $E = (e \in V \mid e^2 = e)_{\vartheta}$. As $e^1 R[e] = e^2 R[e]$ for every 
idempotent $e$, it follows that $E \subseteq (\rd'(V))_{\vartheta}$, so
$$
\rd(E) \subseteq \rd\big((\rd'(V))_{\vartheta}\big)
$$
From proposition \ref{rdce} and (ii) of \cite[Lm.\@ 2.2]{msub}, we deduce that
$$
\rd'(V) \subseteq \rd(E) \qquad \mbox{and} \qquad \rd\big(\rd(E)\big) = \rd(E)
$$
respectively. If $\rd(E)$ is a $\vartheta$-ideal of $\A$, then 
$$
\rd(E) \subseteq \rd\big((\rd'(V))_{\vartheta}\big) \subseteq 
\rd\big((\rd(E))_{\vartheta}\big) = \rd\big(\rd(E)\big) = \rd(E)
$$
If $\rd'(V)$ is a $\vartheta$-ideal of $\A$, then
$$
\rd(E) \subseteq \rd\big((\rd'(V))_{\vartheta}\big) = \rd\big(\rd'(V)\big) \subseteq
\rd\big(\rd(E)\big) = \rd(E)
$$
So $\rd(E) = \rd\big((\rd'(V))_{\vartheta}\big)$ in both cases.
\end{proof}

\noindent
The next theorem gives another definition of $\E_{\vartheta}$.

\begin{theorem} \label{strongE}
Let $\A$ be an associative $R$-algebra. Then $\E_{\vartheta}(\A)$ is the subset of 
uniform $\vartheta$-Mathieu subspaces of $\G_{\vartheta}(\A)$.
\end{theorem}

\begin{proof}
The pre-two-sided case follows from both two-sided cases (take the largest of both $N$'s),
so assume that $\vartheta \ne$ {\it{``pre-two-sided''}}. Take any $V \in \E_{\vartheta}(\A)$ 
and let $a \in \rd(V)$. 

By definition of $\G_{\vartheta}(\A)$, we have 
$a^N \in (e \in V \mid e^2 = e)_{\vartheta}$ for some $N \in \N$. Since (iii) of either lemma 
\ref{lm2.9} or \cite[Lm.\@ 2.9]{msub} tells us that $(e)_{\vartheta} \subseteq V$ for each 
idempotent $e \in V$, we see that $(a^N)_{\vartheta} \subseteq V$. 
So $V$ is a uniform $\vartheta$-Mathieu subspace of $\A$.
\end{proof}

\begin{proposition} \label{coialg}
Let $\A$ be an associative $R$-algebra such that $\A = \rd'(\A)$.
Then each $R$-subspace of $\A$ is contained in $\G_{\vartheta}(\A)$ and
each $\vartheta$-Mathieu subspace of $\A$ is uniform.
\end{proposition}

\begin{proof}
Let $V$ be an $R$-subspace of $\A$. Then $\rd(V) = \rd'(V)$, and on account of 
proposition \ref{rdce}, we have $V \in \G_{\vartheta}(\A)$ by definition of $\G_{\vartheta}$. 
By definition of $\E_{\vartheta}$, it follows from theorem \ref{strongE} that each 
$\vartheta$-Mathieu subspace of $\A$ is uniform.
\end{proof}

\noindent
The rest of this section consists of generalizations of results of section 4 of \cite{msub}.
Just as before, we may omit proofs if those of the original results are already sufficient.
We start with a generalization of \cite[Lm.\@ 4.1]{msub}.

\begin{lemma}[generalizing {\cite[Lm.\@ 4.1]{msub}}] \label{lm4.1}
Let $\A$ be an associative algebra over an Artin ring $R$, and suppose that
$V$ is an $R$-subspace of $\A$. Then for
\begin{enumerate}
 
\item[\upshape(1)] $\A$ is integral (finite) over $R$;
  
\item[\upshape(2)] $V$ is integral (finite) over $R$;

\item[\upshape(3)] every element of $\rd(V)$ is integral over $R$;

\item[\upshape(4)] $\rd(V) = \rd'(V)$;

\item[\upshape(5)] $V \in G_{\vartheta}(\A)$;

\end{enumerate}
we have {\upshape(1)} $\Rightarrow$ {\upshape(2)} $\Rightarrow$ 
{\upshape(3)} $\Rightarrow$ {\upshape(4)} $\Rightarrow$ {\upshape(5)}. 
\end{lemma}

\begin{proof}
By \cite[Th.\@ 8.5]{amd} and \cite[Th.\@ 6.5]{amd}, finite $R$-modules are Noetherian.
By \cite[Prop.\@ 6.2]{amd}, Noetherian $R$-modules are finite. Hence we can replace
(finite) by (Noetherian) in (1) and (2). 
\begin{description}

\item[(1) \imp (2)]
Again by \cite[Prop.\@ 6.2]{amd} we obtain the Noetherian case of (1) $\Rightarrow$ (2).
The integral case of (1) $\Rightarrow$ (2) is trivial.

\item[(2) \imp (3)]
Assume that (2) holds and take any $a \in \rd(V)$. We must show that $a$ is integral over $R$, which 
is the same as that $a^m$ is integral over $R$ for some $m \ge 1$. We can take $m$ 
such that $a^m R[a^m]$ is a subspace of $V$. Now the integral case follows directly.
In the Noetherian case, $R[a^m] = R \cdot 1 + a^m R[a^m]$
is finite because of \cite[Prop.\@ 6.2]{amd}, and $a^m$ is integral over $R$ by \cite[Prop.\@ 5.1]{amd}.

\item[(3) \imp (4)]
This follows directly from theorem \ref{ArtinK}.

\item[(4) \imp (5)]
This follows from proposition \ref{rdce} and the definition of $G_{\vartheta}(\A)$.
\qedhere

\end{description}
\end{proof}

\mathversion{bold}
\subsection{Characterization of $M \in \E_{\vartheta}(\A)$ in terms of idempotents}
\mathversion{normal}

\begin{theorem}[following {\cite[Th.\@ 4.2]{msub}}] \label{th4.2}
Let $V \in \G_{\vartheta}(\A)$. Then $V \in \E_{\vartheta}(\A)$, if and only if 
$(e \in V \mid e^2 = e)_{\vartheta} \subseteq V$.
\end{theorem}

\begin{proof}
Just as in the proof of theorem \ref{strongE}, the pre-two-sided case follows by combining
both one-sided cases. So assume again that $\vartheta \neq$ {\it{``pre-two-sided''}}. 
The `only-if'-part follows from (iii) of either lemma \ref{lm2.9} or \cite[Lm.\@ 2.9]{msub}. 

In order to prove the `if'-part, assume that $(e \in V \mid e^2 = e)_{\vartheta} \subseteq V$,
and take any $a \in \rd(V)$. By definition of $\G_{\vartheta}(\A)$, we have 
$a^N \in (e \in V \mid e^2 = e)_{\vartheta}$ for some $N \in \N$. Hence 
$(a^N)_{\vartheta} \subseteq (e \in V \mid e^2 = e)_{\vartheta} \subseteq V$ by assumption. 
So $V$ is a uniform $\vartheta$-Mathieu subspace of $\A$.
\end{proof}

\begin{corollary}[similar to {\cite[Cor.\@ 4.3]{msub}}] \label{cor4.3}
Let $V \in \G_{\vartheta}(\A)$ such that $V$ does not contain any nonzero idempotent.
Then $V$ is a (uniform) $\vartheta$-Mathieu subspace of $\A$.
\end{corollary}

\noindent
Just as in \cite{msub}, let $I_{\vartheta,V}$ denote the largest 
$\vartheta$-ideal of $\A$ which is contained in $V$ in case 
$\vartheta \ne \mbox{``{\it pre-two-sided}\/''}$, and
$$
I_{\mbox{\scriptsize``{\it pre-two-sided}\/''},V} :=
I_{\mbox{\scriptsize``{\it left}\/''},V} +
I_{\mbox{\scriptsize``{\it right}\/''},V}
$$

\begin{proposition}[similar to {\cite[Prop.\@ 4.5]{msub}}]
Let $V \in \G_{\vartheta}(\A)$ such that $I_{\vartheta,V} = (0)_{\vartheta}$. 
Then $V$ is a (uniform) $\vartheta$-Mathieu 
subspace of $\A$, if and only if $V$ does not contain any nonzero idempotent.

Consequently, if $\rd'(\A) = \A$ and $\A$ has no proper $\vartheta$-ideals
other than $(0)_{\vartheta}$, then any proper $R$-subspace $M$ of $\A$ is a 
(uniform) $\vartheta$-Mathieu 
subspace of $\A$, if and only if $M$ does not contain any nonzero idempotent of $\A$.
\end{proposition}

\begin{proof}
By proposition \ref{coialg}, we can just follow the proof of \cite[Prop.\@ 4.5]{msub}.
\end{proof}

\begin{corollary}[following {\cite[Cor.\@ 4.6]{msub}}]
Let $V$ be an $R$-subspace of an associative $R$-algebra $\A$ and 
$I_V = I_{\mbox{\scriptsize``{\it two-sided}\/''},V}$. Assume that $V \in \G(\A)$ or 
$V/I_V \in \G(\A/I_V)$, where $\G = \G_{\mbox{\scriptsize``{\it two-sided}\/''}}$. 
Then $V$ is a (uniform) Mathieu subspace of $\A$, if and only if $V/I_V$ does not 
contain any nonzero idempotent of the quotient $R$-algebra $\A/I_V$.
\end{corollary}

\begin{proof}
By remark \ref{smrem}, we can just follow the proof of \cite[Cor.\@ 4.6]{msub},
provided that we can prove that $V \in \G(\A)$ implies $V/I_V \in
\G(\A/I_V)$. So let $V \in \G(\A)$. Since
$I_V \subseteq V$, we have 
$$
a^m \in V \Longleftrightarrow\, (a + I_V)^m = a^m + I_V \in V/I_V
$$
Hence $\rd(V) / I_V = \rd(V/I_V)$. The forward implication still holds when we replace
$V$ by any $E \subseteq \A$, so $\rd(E) / I_V \subseteq \rd(E/I_V)$ for any $E \subseteq \A$.
Now take $E = (e \in V \mid e^2 = e)_{\vartheta}$. Then $E/I_V \subseteq 
(e \in V/I_V \mid e^2 = e)_{\vartheta}$. Since $\rd(V) \subseteq \rd(E)$ by definition of $\G$,
we can conclude that
$$
\rd(V/I_V) = \rd(V) / I_V \subseteq \rd(E) / I_V \subseteq \rd(E/I_V) 
\subseteq \rd \big((e \in V/I_V \mid e^2 = e)_{\vartheta}\big)
$$
so that $V/I_V \subseteq \G(\A/I_V)$ by definition of $\G$.
\end{proof}

\begin{proposition}[following {\cite[Prop.\@ 4.7]{msub}} more or less]
Assume that $R$ is local and integrally closed in $\A$. Then every
$V \in \G_{\vartheta}(\A)$ such that $1 \notin V$ is a (uniform) 
$\vartheta$-Mathieu subspace of $\A$.
\end{proposition}

\begin{proof}
Since $R$ is integrally closed in $\A$ and all idempotents 
of $\A$ are integral over $R$, we see that all idempotents of $\A$ 
must lie inside $R \cdot 1 \subseteq \A$. But on account of (i) of \cite[Prop.\@ 1.6]{amd},
$R \cdot 1$ is a local ring. Hence we deduce from lemma \ref{lm3.6} that $\A$ has no
idempotents other than $0$ and $1$.

So if $1 \notin V$, then $V$ does not contain any nonzero idempotent. 
Hence the desired result follows from corollary \ref{cor4.3}.
\end{proof}

\mathversion{bold}
\subsection{Posets of idempotents of $R$-subspaces of $\A$}
\mathversion{normal}

If $K$ is a field, and $\A \ni a$ is a $K$-algebra, then it is clear that $a$ is a 
quasi-idempotent, if and only if $Ka$ contains a nonzero idempotent $e$.
So (2) of \cite[Prop.\@ 4.8]{msub} is equivalent to that 
\begin{enumerate}

\item[(2$'$)] $Ka$ does not have a nonzero idempotent and 
$(e)_{\vartheta} \subseteq Ka$ for every idempotent $e$ of $Ka$.

\end{enumerate}
If (2) (or (2$'$)) does not hold, then $Ka$ contains a nonzero idempotent 
$e \in Ka$, which is unique, and $(e)_{\vartheta} = (a)_{\vartheta}$. 
So (1) of \cite[Prop.\@ 4.8]{msub} can be replaced by 
\begin{enumerate}

\item[(1$'$)] $Ka$ does have a nonzero idempotent and 
$(e)_{\vartheta} \subseteq Ka$ for every idempotent $e$ of $Ka$.

\end{enumerate}
Since $Ka$ has at most one nonzero idempotent, the idempotens of $Ka$
commute with one another.
Hence the following proposition, which additionally shows that $a$ is central in $\A$ 
in case $Ka$ is a (pre-)two-sided Mathieu subspace and $a$ is a quasi-idempotent, is 
indeed a generalization of \cite[Prop.\@ 4.8]{msub}.

\begin{proposition}[generalizing {\cite[Prop.\@ 4.8]{msub}}] \label{prop4.8}
Let $\A$ be an associative $R$-algebra.
Suppose that $V$ is an Artinian $R$-subspace of $\A$, whose idempotents commute
with one another, and let $E := (e \in V \mid e^2 = e)_{\vartheta}$.

Then $V$ is a (uniform) $\vartheta$-Mathieu subspace of $\A$, if and only if $E \subseteq V$.
Furthermore, we have the following if $V$ is indeed a $\vartheta$-Mathieu subspace.
\begin{enumerate}

\item[\upshape(i)] $E$ is a $\bar{\vartheta}$-unital associative algebra over $R$ (with inherited ring 
          operations and a multiplicative $\bar{\vartheta}$-identity that differs from that of $\A$), 
          where $\bar{\vartheta}$ is $\vartheta$ with {\it{``left''}} and {\it{``right''}} interchanged.

\item[\upshape(ii)] If $V$ is a (pre-)two-sided Mathieu subspace of $\A$, then the idempotents
           of $V$ are central in $\A$, and hence $E$ does not depend on the choice of 
           $\vartheta$. In particular, $V$ is a two-sided Mathieu subspace and $E$ is a unital 
           Abelian ring (with inherited ring operations 
           and a multiplicative identity that differs from that of $\A$) in that case.

\end{enumerate}
\end{proposition}

\begin{proof}
Assume that $a \in \rd(V)$, say that $a^m \in V$ for all $m \ge N$. Since the co-integrality of
$a$ is just the decending chain condition on
$$
Ra^N + Ra^{N+1} + Ra^{N+2} + \cdots \supseteq Ra^{N+1} + Ra^{N+2} + \cdots 
\supseteq Ra^{N+2} + \cdots \supseteq \cdots
$$
we see that $a \in \rd'(V)$. Hence $\rd(V) = \rd'(V)$, and by (4) $\Rightarrow$ (5) of lemma \ref{lm4.1}, 
we have $V \in \G_{\vartheta}(\A)$. Therefore, it follows from theorem \ref{th4.2} that $V$ is 
a (uniform) $\vartheta$-Mathieu subspace of $\A$, if and only if $E \subseteq V$. 
So it remains to prove (i) and (ii).
\iftrue
\begin{enumerate}

\item[(i)]
We only need to prove the case $\vartheta =$ {\it{``left''}}, because the case 
$\vartheta =$ {\it{``right''}} is similar and the other two cases follow from (ii). 
So assume that $\vartheta =$ {\it{``left''}}.

Notice that for each idempotent $e$ of $\A$, $\bar{e} := 1 - e$ is another idempotent of $\A$, and
we have $e \bar{e} = 0$.
On account of Zorn's lemma and the descending chain condition of $R$-subspaces of $V$, we can choose 
an idempotent $e \in V$ such that $(\bar{e})_{\vartheta} \cap V$ is minimal. Now take an arbitrary
idempotent $e' \in V$. Then $e'' := e + e' - e e'$ is contained in $V$, and $e''$ is an idempotent
because
$$
\bar{e}'' = \bar{e}\bar{e}' = \bar{e}'\bar{e}
$$
Furthermore,
$(\bar{e}')_{\vartheta} \cap V \supseteq (\bar{e}'')_{\vartheta} \cap V \subseteq 
(\bar{e})_{\vartheta} \cap V$,
and the minimality of $(\bar{e})_{\vartheta} \cap V$ tells us
that $(\bar{e}')_{\vartheta} \cap V \supseteq (\bar{e}'')_{\vartheta} \cap V = 
(\bar{e})_{\vartheta} \cap V$. In particular,
$$
(\bar{e}')_{\vartheta} \supseteq (\bar{e}'')_{\vartheta} \cap V \supseteq 
(\bar{e})_{\vartheta} \cap(e')_{\vartheta}
$$
Consequently, $\bar{e}e' = a \bar{e}'$ for some $a \in \A$, and multiplication by $e'$ gives 
$\bar{e}e' = 0$. Thus $e e' = (1 - \bar{e}) e' = e' - \bar{e}e'  = e'$.
Since $e'$ was arbitrary, we have $e' = e' e = e e'$ for all idempotents $e' \in V$. 
Using that $E$ is a left ideal generated by elements with respect to which $e$ is 
a right identity, we obtain that $E$ is right-unital.  

\item[(ii)] 
Assume that $V$ is a pre-two-sided Mathieu subspace of $\A$. Take any idempotent $e' \in V$ 
and take $a \in \A$ arbitrary. Then $e' + \bar{e}'ae'$ is an idempotent as well,
and by taking $\vartheta =$ {\it{``left''}}, we see that $e' + \bar{e}'ae' \in V$, too.
Since both $e'$ and $e' + \bar{e}' a e'$ are idempotents of $V$, we have
$$
e' + \bar{e}' a e' = (e' + \bar{e}' a e')e' = e'(e' + \bar{e}' a e') = e'
$$
Hence $\bar{e}' a e' = 0$. Adding $e'ae'$ gives $ae' = e'ae'$ and $e'ae' = e'a$ follows in a 
similar manner. 

So every idempotent $e' \in V$ is central in $\A$, and therefore $E$ does not depend on 
$\vartheta$. In particular, $E$ is a two-sided ideal of $\A$, so $V$ has to be a 
(uniform) two-sided Mathieu subspace of $\A$. Furthermore, the right identity $e$ of $E$ that 
we get by taking $\vartheta =$ {\it{``left''}} in (i), is a two-sided multiplicative identity 
of $E$. 
\qedhere

\end{enumerate}
\else
This is left as an exercise to the reader. \qedhere
\fi
\end{proof}

\noindent
\iftrue
It is well-known that the idempotents of unital rings are central once they commute relatively
(or with all nilpotent elements). The proof of that is based on the idempotence of $e' + \bar{e}' a e'$
(the nilpotence of $\bar{e}' a e'$), so the idea to use that idempotent in the above proof was obvious.

\fi
The assumption that the idempotents of $V$ in proposition \ref{prop4.8} commute relatively 
ensures that they form 
a lattice with respect to $e \wedge e' := ee'$, $e \vee e' := e + e' - ee'$, and $(e \le e') := 
(e = ee' = e'e)$ (the idempotent property is just the reflexivity of $\le$).
\iftrue
From the above proof, we can deduce that in some cases, that lattice has a 
top element.

To obtain the existence of a top element, it is however not needed to make any assumptions on a certain
lattice structure on the poset of idempotents whose ordering is given by $(e \le e') := (e = ee' = e'e)$.
More generally, one can even 
\else
The reader may 
\fi
show that the lattice of idempotents of $V$ must be complete, by proving the following.

\begin{proposition}
Let $\A$ be an associative $R$-algebra and $L$ be a set of idempotents of
$\A$, which is a poset with respect to $(e \le e') := (e = ee' = e'e)$. 
Suppose that every commutative $R$-subspace, generated by the multiplicative closure
of a subset of $L$, is either Noetherian or Artinian over $R$. Then we have the following.
\begin{enumerate}

\item[\upshape(i)] Every chain of $L$ has both a minimum and a maximum element.

\item[\upshape(ii)] If $L$ admits a lattice structure, then 
for every $S \subseteq L$ there exist a finite $S' \subseteq S$ such that
$\bigwedge S = \bigwedge S'$ and $\bigvee S = \bigvee S'$.
In particular, every lattice structure over $L$ is complete.

\end{enumerate}
\end{proposition}
\iftrue

\begin{proof}
We only prove the Noetherian case here. Using ideas in the proof of \ref{prop4.8}, the reader may 
treat the Artinian case himself. Again, we write $\bar{e} := 1 - e$ for idempotents $e$.
\begin{enumerate}

\item[(i)] Let $S$ be a chain of $L$. Since $e \le e'$ implies $e e' = e' e$, we see
that $R[S]$ is commutative. Consequently, $V := S R[S]$ is Noetherian over $R$ and hence over
$R[S]$ as well by assumption. So we can take $e^{\vee}, e^{\wedge} \in S$ such that 
$e^{\vee} V$ and $\bar{e}^{\wedge}V$ are maximal.

Suppose that there exists an $e \in S$ such that $e \nless e^{\vee}$. Then $e \ge e^{\vee}$, thus
$e e^{\vee} = e^{\vee}$. Hence $e V \supseteq e^{\vee} V$, so $e V = e^{\vee} V$
by definition of $e^{\vee}$. Multiplication of $e V = e^{\vee} V$ by $\bar{e}^{\vee}$ and 
$\bar{e}$ respectively gives $\bar{e}^{\vee} e V = 0 = \bar{e} e^{\vee} V$. By taking the elements
$e$ and $e^{\vee}$ of $V$ respectively, we see that $\bar{e}^{\vee} e = 0 = \bar{e} e^{\vee}$.
Adding $e^{\vee}e$ subsequently gives $e = e^{\vee}$. Hence $e \leq e^{\vee}$ for all $e \in S$, i.e.\@ 
$e^{\vee}$ is maximum.

Suppose next that there exists an $e \in S$ such that $e \ngtr e^{\wedge}$. Then $e \le e^{\wedge}$, 
thus $e = e e^{\wedge}$. Multiplication by $\bar{e}^{\wedge}$ gives $e\bar{e}^{\wedge} = 0$,
and adding $\bar{e}\bar{e}^{\wedge}$ subsequently gives $\bar{e}^{\wedge} = \bar{e}\bar{e}^{\wedge}$.
Consequently, $\bar{e}^{\wedge}V \subseteq \bar{e}V$, so $\bar{e}^{\wedge}V = \bar{e}V$ by definition of 
$e^{\wedge}$. Multiplication of $\bar{e}^{\wedge}V = \bar{e}V$ by $e$ and $e^{\wedge}$ respectively
gives $\bar{e}^{\wedge}eV = 0 = \bar{e}e^{\wedge}V$. Now a similar argument as in the previous paragraph
tells us that $e = e^{\wedge}$. Hence $e \geq e^{\wedge}$ for all $e \in S$, i.e.\@ $e^{\wedge}$ is minimum.

\item[(ii)] Assume that $L$ admits a lattice structure. Take any subset $S$ of $L$.
On account of (i) and Zorn's lemma, we can choose $e^{\wedge}, e^{\vee} \in S$ which are minimal
and maximal respectively. If $S$ is closed under dyadic $\wedge$ or $\vee$, then
$e^{\wedge}$ and $e^{\vee}$ have to be minimum and maximum respectively in $S$ 
(see the proof of proposition \ref{prop4.8}). In general, we can take the closure $\bar{S}$
of $S$ under dyadic $\wedge$ or $\vee$, and obtain that $e^{\wedge}$ and $e^{\vee}$ are a 
finite meet and join of elements of $S$ respectively (just like any element of $\bar{S}$). 
Furthermore, we can take for $S'$ the union of both underlying finite sets. \qedhere

\end{enumerate}
\end{proof}
\fi

\mathversion{bold}
\subsection{Radicals of uniform $\vartheta$-Mathieu subspaces in terms of 
radicals of $I_{\vartheta,M}$}
\mathversion{normal}

\begin{lemma}[following {\cite[Lm.\@ 4.9]{msub}}] \label{lm4.9}
Let $\A$ be an associative $R$-algebra and $M$ a $\vartheta$-Mathieu subspace of $\A$. Then
$\rd'(M) = \rd'(I_{\vartheta,M})$.

If $M$ is even a uniform $\vartheta$-Mathieu subspace of $\A$, then
$\rd(M) = \rd(I_{\vartheta,M})$.
\end{lemma}

\begin{proof}
The claim $\rd'(M) = \rd'(I_{\vartheta,M})$ for $\vartheta$-Mathieu subspaces 
$M$ of $\A$ follows from the last claim of theorem \ref{th3.89}. 
The last assertion of lemma \ref{lm4.9} follows in a similar manner from the 
definition of uniform $\vartheta$-Mathieu subspace.
\end{proof}

\begin{theorem}[following {\cite[Th.\@ 4.10]{msub}} more or less] \label{th4.10}
Let $M$ be an $R$-subspace of $\A$. Then the following statements are equivalent.
\begin{enumerate}

\item[\upshape(1)] $M$ is a uniform $\vartheta$-Mathieu subspace of $\A$,

\item[\upshape(2)] $\rd(M) = \rd(I_{\vartheta,M})$.

\item[\upshape(3)] For every $R$-subspace $V$ of $\A$ such that 
$I_{\vartheta,M} \subseteq V \subseteq M$, $V$ is a uniform 
$\vartheta$-Mathieu subspace of $\A$ and $\rd(V) = \rd(I_{\vartheta,M})$.

\end{enumerate}
\end{theorem}

\begin{proof}
Since (1) $\Rightarrow$ (2) follows from 
lemma \ref{lm4.9} and (3) $\Rightarrow$ (1) is trivial, we assume (2) to show 
$(2) \Rightarrow (3)$. Let $V$ be an $R$-subspace $V$ of $\A$ such that 
$I_{\vartheta,M} \subseteq V \subseteq M$. Then $\rd(V) = \rd(I_{\vartheta,M})$
because of (2).

Take $a \in \rd(V)$ arbitrary.
Then $a \in \rd(I_{\vartheta,M})$ and there exists an $N \in \N$ such that 
$a^N \in I_{\vartheta,M}$ and $(a^N)_{\vartheta} \subseteq I_{\vartheta,M} 
\subseteq V$. Since $a$ was arbitrary, we conclude that $V$ is a 
uniform $\vartheta$-Mathieu subspace of $\A$, and (3) follows.
\end{proof}

\begin{corollary}[following {\cite[Cor.\@ 4.11]{msub}}]
Let $\A$ be a $\vartheta$-simple associative $R$-algebra (i.e.\@ $\A$ has no 
proper $\vartheta$-ideals other than zero) and $M$ 
a proper uniform $\vartheta$-Mathieu subspace of $\A$. 
Then $\rd(M) = \rd((0)_{\vartheta})$ and all $R$-subspaces 
$V \subseteq M$ are uniform (two-sided) Mathieu subspaces of $\A$.
\end{corollary}

\begin{proof}
Since $\A$ is $\vartheta$-simple, we have $I_{\vartheta,M} = (0)_{\vartheta}$.
By theorem \ref{th4.10}, we get $\rd(M) = \rd((0)_{\vartheta})$. Since 
$(0)_{\vartheta} \subseteq V$ for all $R$-subspaces $V$ of $\A$, 
we additionally deduce from theorem \ref{th4.10} that all $R$-subspaces $V \subseteq M$ 
are uniform two-sided Mathieu subspaces of $\A$.
\end{proof}

\noindent
If $\A$ is Abelian, then the definition of $\G_{\vartheta}(\A)$ does not depend on 
$\vartheta$. For that reason, we simply write $\G(\A)$ in that case.
The assumption that $\A$ is commutative in \cite[Th.\@ 4.12]{msub} can be weakened (or be
replaced by that $\A$ is reduced to ensure that idempotents are central), 
but not without eliminating the `commutative fact' that radicals of ideals are ideals 
themselves.

\begin{theorem}[generalizing {\cite[Th.\@ 4.12]{msub}}] \label{th4.12}
Let $\A$ be an Abelian associative $R$-algebra and $V \in \G(\A)$. If $\vartheta \ne$
{\it{``pre-two-sided''}}, then $V$ is a (uniform) 
($\vartheta$\discretionary{-)}{}{-)}Mathieu subspace of $\A$, 
if and only if $\rd(V)$ is the radical of some $\vartheta$-ideal of $\A$, which is the 
case when $\rd(V)$ is a $\vartheta$-ideal itself.
\end{theorem}

\begin{proof}
Assume that $\vartheta \ne$ {\it{``pre-two-sided''}}.
The `only if'-part follows directly from lemma \ref{lm4.9}. To prove the `if'-part 
and the last claim along with it, suppose that $\rd(V) \in \{J, \rd(J)\}$ for some $\vartheta$-ideal
$J$ of $\A$. Since $\A$ is abelian, we obtain that $(e)_{\vartheta}$ does not depend on $\vartheta$.
Hence by theorem \ref{th4.2}, it suffices to show that $(e)_{\vartheta} \subseteq V$
for every idempotent $e \in V$.

So take any idempotent $e \in V$. Since the idempotents of any subset of $\A$ coincides with
those of its radical, we have $e \in J$ in any case (both when $\rd(V) = J$ and when $\rd(V) = \rd(J)$). 
Using that $(e)_{\vartheta}$ and $J$ are $\vartheta$-ideals, we deduce that 
$(e)_{\vartheta}^m \subseteq (e)_{\vartheta} \subseteq J$ for every $m \ge 1$. 
Hence $(e)_{\vartheta} \subseteq \rd(V)$ in any case (both when $\rd(V) = J$ and when $\rd(V) = \rd(J)$).
So for arbitrary $b \in \A$, we have $e b^m = (eb)^m \in V$ when $m \gg 0$. 
Consequently, for the $R$ subspace $V_e$ of $\A$ consisting of elements $a \in \A$ such that 
$ea \in V$, we have $\rd(V_e) = \A$. Therefore, we have $V_e = \A$ on account of
\cite[Lm.\@ 2.4]{msub}, i.e.\@ $(e)_{\vartheta} \subseteq V$.
\end{proof}

\begin{corollary}[following {\cite[Cor.\@ 4.13]{msub}} more or less] \label{cor4.13}
Let $\A$ be an Abelian associative $R$-algebra and $V  \in \G(\A)$.  If $\vartheta \ne$
{\it{``pre-two-sided''}} and $\rd(V)$ is a $\vartheta$-ideal, then $\rd(V)$ is both a 
two-sided ideal itself and the radical of some two-sided ideal.
\end{corollary}
 
\subsection{Unions and intersections of Mathieu subspaces $M \in \E_{\vartheta}(\A)$}

\begin{proposition}[combining {\cite[Prop.\@ 4.16]{msub}} and {\cite[Prop.\@ 4.18]{msub}}]
Let $M_i$ ($i \in I$) be a family of proper $\vartheta$-Mathieu subspaces of an
associative $R$-algebra $\A$. 
\begin{enumerate}

\item[\upshape(i)] If\/ $\bigcap_{i \in I} M_i \in \G_{\vartheta}(\A)$,
then $\bigcap_{i \in I} M_i$ is a proper (uniform) $\vartheta$-Mathieu subspace of $\A$.

\item[\upshape(ii)] If\/ $\bigcup_{i \in I} M_i \in \G_{\vartheta}(\A)$,
then $\bigcup_{i \in I} M_i$ is a proper (uniform) $\vartheta$-Mathieu subspace of $\A$.

\end{enumerate}
More generally, suppose that $J$ is a set of subsets of $I$, and define 
$$
M_J := \bigcup_{j \in J} \bigcap_{i \in j} M_i
$$
\begin{enumerate}

\item[\upshape(iii)] If $M_J \in \G_{\vartheta}(\A)$, then $M_J$ is a proper (uniform) 
$\vartheta$-Mathieu subspace of $\A$.

\end{enumerate}
\end{proposition}

\noindent
Using proposition \ref{coialg} instead of \cite[Lm.\@ 4.1]{msub}, we obtain the following from
the paragraph that precedes \cite[Prop.\@ 4.20]{msub}.

\begin{proposition}[similar to {\cite[Prop.\@ 4.20]{msub}}] \label{prop4.19}
Let $\A$ be an associative $R$-algebra such that $\A = \rd'(\A)$.
If $V$ is an $R$-subspace of $\A$, then the following statements hold.
\begin{enumerate}

\item[\upshape(i)] There exists at least one $\vartheta$-Mathieu subspace
which is maximal among all the $\vartheta$-Mathieu subspaces of $\A$
contained in $V$.

\item[\upshape(ii)] There exists a unique $\vartheta$-Mathieu subspace
which is minimum among all the $\vartheta$-Mathieu subspaces of $\A$
containing $V$, namely the intersection of all $\vartheta$-Mathieu subspaces 
of $\A$ that contain $V$.

\item[\upshape(iii)] Let $\F$ be the collection of proper $\vartheta$-Mathieu subspaces
$M$ of $\A$ such that $M \supseteq V$. If $\F \ne \varnothing$,
then $\F$ has at least one maximal element and a unique minimum element.

\end{enumerate}
\end{proposition}

\begin{theorem}[following {\cite[Th.\@ 4.21]{msub}}]
Let $\A$ be an associative $R$-algebra such that $\rd'(\A) = \A$.
Then every proper $\vartheta$-Mathieu subspace of $\A$ is contained in a 
maximal proper $\vartheta$-Mathieu subspace of $\A$.

In particular, $\A$ has at least one maximal proper $\vartheta$-Mathieu subspace.
If this $\vartheta$-Mathieu subspace cannot be taken nonzero, then $R \cdot 1 = \A$ 
and $\A$ is a field.
\end{theorem}

\begin{proof}
Except for the last claim, we can just follow the proof of \cite[Th.\@ 4.21]{msub}.
So assume that $\A$ does not have a nonzero proper $\vartheta$-Mathieu subspace.
By \cite[Th.\@ 6.2]{msub}, $\A$ is a field which is isomorphic to the fraction field of 
$R \cdot 1$, which in turn is an integral domain. 

Hence it suffices to show that $R \cdot 1$ is a field itself.
So let $a \in R \cdot 1$. By $\rd'(\A) = \A$, $a$ is co-integral over $R$. Now the desired 
result follows from (ii) of proposition \ref{prop3.3}, because $R \cdot 1$ is a domain.
\end{proof}

\begin{corollary}[similar to {\cite[Cor.\@ 4.22]{msub}}]
Let $V$ be an $R$-subspace of an associative $R$-algebra $\A$ such that 
$\A = \rd'(\A)$, and assume that the $\vartheta$-ideal generated by $V$ 
is not the whole algebra $\A$. Then there exists a maximal nonzero
$\vartheta$-Mathieu subspace $M$ of $\A$ such that $V \subseteq M$.
\end{corollary}

\section{Strongly simple algebras}

Our starting point is the following theorem, which is the main
theorem of section 6 of \cite{msub}.

\begin{theorem}[following {\cite[Th.\@ 6.2]{msub}}]
Assume $R$ is a nontrivial commutative ring and $\A$ a nontrivial 
associative $R$-algebra such that $\A$ has no proper nonzero $\vartheta$-Mathieu 
subspaces. Then $R \cdot 1$ is an integral domain and $\A$ is isomorphic to
the field of fractions of $R \cdot 1$.
\end{theorem}

\noindent
The following definition appears at the beginning of section 6 in \cite{msub},
which has the same title as this section. The above theorem says that variants
of the below definition with $\vartheta \ne$ ``{\it two-sided}\/'' are unnecessary,
since $\A$ is a field regardless of what $\vartheta$ is.

\begin{definition}
Let $R$ be a commutative ring and $\A$ an associative $R$-algebra. We
say that $\A$ is strongly simple if $\A$ has no proper nonzero
Mathieu subspaces over $R$.
\end{definition}

\noindent
Since $R$-algebras are $(R \cdot 1)$-algebras and strongly simple algebras over $R$ can only be 
fraction fields of $R \cdot 1$, we restrict to integral domains $R$ with fraction field $\A = K$ 
from now on.

In an earlier version of \cite{msub}, the author conjectured that
the only strongly simple algebras would be fields over theirselves. But this
is not true. The main result of this section is the following.

\begin{theorem} \label{stronglysimple}
Let $R$ be an integral domain with fraction field $K$. Then the 
following statements hold.
\begin{enumerate}

\item[\upshape(i)] Each prime ideal $\p \subset R$ of height one is of the form 
$\p = M \cap R$ for some Mathieu subspace $M$ of $K$ over $R$.

\item[\upshape(ii)] If $r \in R$ is anti-Archimedean, i.e.\@ 
$\bigcap_{n = 1}^{\infty} r^n R \ne (0)$, 
then there is no proper Mathieu subspace of 
$K$ over $R$ that contains $r$.

\end{enumerate}
\end{theorem}

\begin{corollary}
Let $R$ be an integral domain with fraction field $K$. Then for the following 
statements:
\begin{enumerate}

\item[\upshape(1)] $R$ is anti-Archimedian, i.e.\@ every nonzero $r \in R$
is anti-Archimedian;

\item[\upshape(2)] $K$ is strongly simple over $R$;

\item[\upshape(3)] $R$ does not have a prime ideal of height one;

\end{enumerate}
we have {\upshape(1)} $\Rightarrow$ {\upshape(2)} $\Rightarrow$ {\upshape(3)}.
\end{corollary}

\begin{listproof} 
\begin{description}

\item[(1) \imp (2)]
Since every nonzero Mathieu subspace of $K$ over $R$ 
contains a nonzero element of $R$ (the numerator), the desired result follows 
from (ii) of theorem \ref{stronglysimple}.

\item[(2) \imp (3)]
This follows from (i) of theorem \ref{stronglysimple}. \qedhere

\end{description}
\end{listproof}

\begin{listproof}[Proof of theorem \ref{stronglysimple}.]
\begin{enumerate}

\item[(i)] Fix a prime ideal of height one of $R$ and let $S$ be the set of 
all elements of $R$ that are not contained in the prime ideal at hand.
By replacing $R$ by the localization $S^{-1}R$, we may assume that $R$ is a local ring of 
dimension 1. Then it is known that there exists a valuation ring $D$ 
(with the same fraction field $K$) that dominates $R$, i.e.\@ $D$ contains $R$ 
and $R \cap \m_D = \m_R$, where $\m_R$ denotes the maximal ideal of the local ring $R$.
 
Assume $\q$ is a nonzero prime ideal of $D$. Since $K$ is the fraction field of $R$, the 
contraction $\q \cap R$ of $\q$, which is a prime ideal, is nonzero. 
But $R$ has dimension one, so $\q \cap R = \m_R$. Now define $M$ as the intersection 
of all nonzero prime ideals $\q$ of $D$. Then $M \cap R = \m_R$.
It is known that the ideals of $D$ are totally ordered by 
inclusion, from which we can deduce that $M$ is a prime ideal of $D$. 

So $M$ is a prime ideal of height one. From lemma \ref{valht1} below,
it follows that $M$ is a Mathieu subspace of $K$ over $D \supseteq R$.

\item[(ii)] Assume $r \in R$ is not Archimedian. Then there exists a nonzero $a \in R$ such 
that $r^n \mid a$ for all $n \in \N$. Let $M \ni r$ be a Mathieu subspace of $K$ over $R$. 
Then $r^n/a \in M$ for some $n \in \N$. Since $r^n \mid a$, we have 
$r^n b = a$ for some $b \in R$. It follows that $1 = b\, r^n / a \in M$. 
Thus $M = K$ by \cite[Cor.\@ 2.10]{msub}. \qedhere

\end{enumerate}
\end{listproof}

\begin{lemma} \label{valht1}
Assume $D$ is a valuation ring with maximal ideal $\m_{D}$ and fraction field $K$. 
Then any prime ideal of height one of $D$ is a Mathieu subspace of $K$ over $D$.
\end{lemma}

\begin{proof}
Let $\p$ be a prime ideal of height one of $D$.
Assume that $\p$ is not a Mathieu subspace of $K$ over $D$, say that
$a \in \p$ and $b \in K$ such that $a^{n+1} b \notin \p$ for infinitely many 
$n \in \N$. Write $b = t/d$ with $t,d \in D$. Since $ta \in \p$, we have
$a^n /d \notin D$ for infinitely many $n \in \N$. 
Thus $v(d) > v(a^n) = n\, v(a)$ for infinitely many $n \in \N$. 

Hence the ideal 
$\q$ of $D$ consisting of elements $c$ such that $v(c) > n v(a)$ for 
infinitely many $n \in \N$ is nonzero. In fact, it is a prime ideal, since 
$v(c_1 c_2) > 2n v(a)$ (for infinitely many $n \in \N$)
implies $v(c_1) > n v(a)$ or $v(c_2) > n v(a)$ (for infinitely many $n \in \N$).
Since $\q \subsetneq \p$, we have a contradiction with the height one 
assumption on $\p$, so $\p$ is a Mathieu 
subspace of $K$ over $D$.
\end{proof}

\begin{example}
Let $k$ be a field. The valuation domain
$$
D = k[[x_1,x_2,x_3,\ldots]]
\left[\frac{x_2}{x_1},\frac{x_2}{x_1^2},\frac{x_2}{x_1^3},\ldots\right]
\left[\frac{x_3}{x_2},\frac{x_3}{x_2^2},\frac{x_3}{x_2^3},\ldots\right]\cdots
$$
with fraction field
$$
K = k((x_1,x_2,x_3,\ldots)) = k[[x_1,x_2,x_3,\ldots]][x_1^{-1},x_2^{-1},x_3^{-1},\ldots]
$$
has value group $\Z[t]$, ordered by $f(t) > g(t) \Leftrightarrow \lc(f(t)-g(t)) > 0$.
Here, $\lc(f(t))$ denotes the leading coefficient of $f$ with respect to $t$. 

The value function $v$ is defined by 
$$
v(a) = \min \{\deg_{x_1} s + t \deg_{x_2} s + t^2 \deg_{x_3} s + \cdots
\mid s \mbox{ is a term of } a\}
$$
where $\min$ is with respect to the above ordering on $\Z[t]$.
So if $a$ has a term $s = x_1^{\alpha_1} x_2^{\alpha_2} \cdots x_n^{\alpha_n}$,
and for every other term $s'$ of $a$, $s'\!/s$ has positive degree in the variable
of largest index which appears in $s'\!/s$, i.e. in which $s'\!/s$ has nonzero degree,
then $v(a) = \alpha_n t^{n-1} + \alpha_{n-1} t^{n-2} + \cdots + \alpha_1$.  

If $g(t) \in \Z[t]$, then there exists a $h(t) \in \Z[t]$ with positive leading
coefficient such that $\deg h > \deg g$. Hence $n g(t) < h(t)$ for all $n \in \N$. 
On account of (1) $\Rightarrow$ (5) of proposition \ref{valring} below, 
$K$ is strongly simple over $D$.
\end{example}

\noindent
The following proposition classifies the valuation rings that are strongly simple.

\begin{proposition} \label{valring}
Assume $D$ is a valuation ring with maximal ideal $\m_{D}$ and fraction field $K$. 
Then the following statements are equivalent.
\begin{enumerate}

\item[\upshape(1)] $D$ is not strongly simple,

\item[\upshape(2)] $D$ has a prime ideal of height $1$,

\item[\upshape(3)] $D$ has a nonzero element that is contained in every nonzero prime 
ideal of $D$,

\item[\upshape(4)] $D$ has a nonzero Archimedian element, i.e.\@ an element $a$ such that
$$\bigcap_{n = 1}^{\infty} D a^n = (0)$$

\item[\upshape(5)] The value group G of $D$ has an element $g$ such that for all $h \in G$,
there exists an $n \in \N$ such that $ng > h$.

\end{enumerate}
\end{proposition}

\iftrue
\begin{listproof}
\begin{description}
\item [(1) \imp (5)] Assume $D$ is not strongly simple. Let $M$ be a proper nonzero
Mathieu subspace of $K$ over $D$. If $M \nsubseteq \m_D$, then $M$ has an element $a$
such that $v(a) \le 0$ and thus $1/a \in D$, whence $1 = 1/a \cdot a \in M$. 
This contradicts \cite[Cor.\@ 2.10]{msub}, so $M \subseteq \m_D$. 

Let $a \in M$ be nonzero and $b \in K$. Take $n \in \N$ such that
$a^n / b \in M$. Since $M \subseteq \m_D$, we have 
$n \, v(a) - v(b) = v(a^n/b) > 0$. Thus $g = v(a)$ suffices.

\item[(2) \imp (1)] This follows from lemma \ref{valht1}.

\item[(3) \imp (2)] 
Assume $a$ is nonzero and contained in every nonzero prime ideal of $D$. 
Let $\p$ be the radical of
$(a)$. Then $\p$ is contained in every nonzero prime ideal of $D$, thus 
it suffices to show that $\p$ is prime. So assume $c_1 c_2 \in \p$. 
Then $a \mid c_1^n c_2^n$ for some $n$, 
thus $v(a) < n v(c_1) + n v(c_2)$. Hence $v(a) < 2n v(c_1)$ or $v(a) < 2n v(c_2)$.\
In the first case, $a \mid c_1^{2n}$ and thus $c_1 \in \p$. In the second case,
$c_2 \in \p$. Thus $\p$ is prime, and its height is equal to one.

\item[(4) \imp (3)] 
Assume $a$ is a nonzero Archimedian element of $D$. 
Assume $\q$ is a nonzero prime ideal and $b \in \q$. Since $a$ is Archimedian,
we have $a^n \nmid b$ for some $n \in \N$. Since $D$ is a valuation ring,
$b \mid a^n$ follows. Consequently $a \in \q$, as desired.

\item[(5) \imp (4)]
Assume (5) and that $g$ is as in (5). Take $a$ such that $v(a) = g$. 
Then for all $b \in D$, there exists an $n \in \N$ such that 
$v(a^n) = ng > v(b)$. Hence $a^n \nmid b$. This gives
the desired result. \qedhere

\end{description}
\end{listproof}
\else
\noindent
The proof is left as an exercise to the reader.
\fi

\end{document}